\documentclass[a4paper]{amsart}
\usepackage{amssymb,latexsym}
\usepackage{enumerate}
\usepackage{xypic}
\usepackage{amsmath}
\allowdisplaybreaks[1]

\makeatletter
\@namedef{subjclassname@2010}{%
  \textup{2010} Mathematics Subject Classification}
\makeatother

\newtheorem{thm}{Theorem}[section]
\newtheorem{cor}[thm]{Corollary}
\newtheorem{prop}[thm]{Proposition}
\newtheorem{lem}[thm]{Lemma}

\theoremstyle{definition}
\newtheorem{rem}[thm]{Remark}

\numberwithin{equation}{section}

\newcommand{\pt}{{x}}
\newcommand{\ptx}{{x}}

\newcommand{\ulb}{{\textup{(}}}
\newcommand{\urb}{{\textup{)}}}

\newcommand{\T}{{\mathbb{T}}}
\newcommand{\C}{{\mathbb{C}}}

\newcommand{\Z}{{\mathbb{Z}}}
\newcommand{\Ccross}{\C^\times}

\newcommand{\Eone}{{E_1}}
\newcommand{\Estar}{{E_*}}
\newcommand{\Eprimestar}{{E^\prime_*}}
\newcommand{\Fprimestar}{{F^\prime_*}}
\newcommand{\Eprimeone}{{E^\prime_1}}
\newcommand{\Fprimeone}{{F^\prime_1}}

\newcommand{\interior}{{\textup{int}}}
\newcommand{\intsymbol}{{\circ}}

\newcommand{\cesaro}{\sigma^C}
\newcommand{\homeo}{{\sigma}}
\newcommand{\ch}{{\omega}}
\newcommand{\st}{{\phi}}

\newcommand{\surjmap}{{\psi}}
\newcommand{\rest}{{\restriction}}
\newcommand{\topspace}{{X}}

\newcommand{\dynsys}{(\topspace, \homeo)}
\newcommand{\dynsysshort}{{\Sigma}}
\newcommand{\lone}{{\ell^1(\dynsysshort)}}
\newcommand{\cstar}{{C^*(\dynsysshort)}}
\newcommand{\fs}{{c_{00}(\dynsysshort)}}

\newcommand{\coeffalg}{{C(\topspace)}}
\newcommand{\loneposcone}{{\lone_+}}
\newcommand{\cstarcomm}{{\coeffalg^\prime_\ast}}
\newcommand{\lonecomm}{{\coeffalg_1^\prime}}
\newcommand{\fscomm}{{\coeffalg_{00}^\prime}}

\newcommand{\maxidealspacelonecomm}{{\Delta(\lonecomm)}}
\newcommand{\maxidealspacecstarcomm}{{\Delta(\cstarcomm)}}

\newcommand{\per}{{\textup{Per}}}
\newcommand{\fix}{{\textup{Fix}}}

\newcommand{\fixk}{{{\fix}_k(\homeo)}}
\newcommand{\fixl}{{{\fix}_l(\homeo)}}
\newcommand{\fixm}{{{\fix}_m(\homeo)}}
\newcommand{\fixn}{{{\fix}_n(\homeo)}}

\newcommand{\fixq}{{{\fix}_q(\homeo)}}
\newcommand{\fixs}{{{\fix}_s(\homeo)}}
\newcommand{\perk}{{{\per}_k(\homeo)}}

\newcommand{\perq}{{{\per}_q(\homeo)}}

\newcommand{\perpnew}{{{\per}_p(\homeo)}}

\newcommand{\fixkint}{{{\fix}_k^\intsymbol (\homeo)}}
\newcommand{\fixlint}{{{\fix}_l^\intsymbol (\homeo)}}

\newcommand{\fixnint}{{{\fix}_n^\intsymbol (\homeo)}}

\newcommand{\fixqint}{{{\fix}_q^\intsymbol (\homeo)}}
\newcommand{\fixsint}{{{\fix}_s^\intsymbol (\homeo)}}

\newcommand{\pernint}{{{\per}_n^\intsymbol (\homeo)}}
\newcommand{\perqint}{{{\per}_q^\intsymbol (\homeo)}}

\newcommand{\perpnewint}{{{\per}_p^\intsymbol (\homeo)}}

\newcommand{\aperpoints}{{\textup{Aper}(\homeo)}}
\newcommand{\perpoints}{{\per(\homeo)}}

\newcommand{\intfixedpoints}{{\bigcup_{q\geq 1}\fixqint}}
\newcommand{\intperpoints}{{\bigcup_{q\geq 1}\perqint}}

\newcommand{\supp}{{\textup{supp\,}}}
\newcommand{\ev}{\textup{ev}}

\begin{document}

\title[Maximal abelian subalgebras and projections]{Maximal abelian subalgebras and projections in two Banach algebras associated with a topological dynamical system}

\author[M.\ de Jeu]{Marcel de Jeu}
\address{Marcel de Jeu, Mathematical Institute, Leiden University, P.O.\ Box 9512, 2300 RA Leiden, The Netherlands}
\email{mdejeu@math.leidenuniv.nl}

\author[J.\ Tomiyama]{Jun Tomiyama}
\address{Jun Tomiyama, Department of Mathematics, Tokyo Metropolitan University, Minami-Osawa, Hachioji City, Japan}
\email{juntomi@med.email.ne.jp}

\date{}

\begin{abstract}
If $\dynsysshort=\dynsys$ is a topological dynamical system, where $\topspace$ is a compact Hausdorff space and $\homeo$ is a homeomorphism of $\topspace$, then a crossed product Banach $\sp{\ast}$-algebra $\lone$ is naturally associated with these data. If $X$ consists of one point, then $\ell^1(\dynsysshort)$ is the group algebra of the integers. The commutant $\lonecomm$ of $\coeffalg$ in $\lone$ is known to be a maximal abelian subalgebra which has non-zero intersection with each non-zero closed ideal, and the same holds for the commutant $\cstarcomm$ of $\coeffalg$ in $\cstar$, the enveloping $C^*$-algebra of $\lone$. This intersection property has proven to be a valuable tool in investigating these algebras. Motivated by this pivotal role, we study $\lonecomm$ and $\cstarcomm$ in detail in the present paper. The maximal ideal space of $\lonecomm$ is described explicitly, and is seen to coincide with its pure state space and to be a topological quotient of $\topspace\times\T$. We show that $\lonecomm$  is hermitian and semisimple, and that its enveloping $C^*$-algebra is $\cstarcomm$. Furthermore, we establish necessary and sufficient conditions for projections onto $\lonecomm$ and $\cstarcomm$ to exist, and give explicit formulas for such projections, which we show to be unique. In the appendix, topological results for the periodic points of a homeomorphism of a locally compact Hausdorff space are given.
\end{abstract}



\maketitle

\section{Introduction}\label{s:introduction}

Let $\dynsysshort=(\topspace,\homeo)$ be a topological dynamical system, where $\topspace$ is a compact Hausdorff space, and $\homeo$ is a homeomorphism of $\topspace$. Via $\homeo$, the integers $\Z$ act on $\coeffalg$, and this action induces a twisted convolution on $\ell^1(\Z,\coeffalg)$, thus obtaining a unital involutive Banach algebra of crossed product type, denoted by $\lone$, and with enveloping $C^*$-algebra $\cstar$. The relation between the structure of $\cstar$ and the dynamics of $\dynsys$ is rather well studied (see, e.g., \cite{tomiyama_book, tomiyama_notes_one, tomiyama_notes_two} for an introduction), but more recently the algebra $\lone$ itself, which is the involutive Banach algebra most naturally associated with $\dynsysshort$, has also been investigated in \cite{de_jeu_svensson_tomiyama}. As it turns out, several results which are known for $\cstar$ have an analogue for $\lone$. For example, both $\lone$ and $\cstar$ have only trivial closed ideals, or only trivial self-adjoint closed ideals, precisely when $\topspace$ has an infinite number of points and $\dynsysshort$ is a minimal dynamical system. As a further example, both $\lone$ and $\cstar$ are prime, precisely when $\topspace$ has an infinite number of points and $\dynsysshort$ is topologically transitive. In spite of the formal similarity of such results, the basic underlying proofs for $\lone$ in \cite{de_jeu_svensson_tomiyama} are quite different from the known proofs for $\cstar$. There are, in fact, also differences at a structural level. For example, all closed ideals of $\cstar$ are naturally self-adjoint, but for $\lone$ this need not be so: all closed ideals of $\lone$ are self-adjoint precisely when $\dynsysshort$ is free. Such differences are only to be expected: if, for example, $\topspace$ consists of one point, then $\lone=\ell^1(\Z)$, and $\cstar=C(\T)$, and certainly $\ell^1(\Z)$ has a rather different, more complicated, structure than $C(\T)$.

The algebra $\coeffalg$ is embedded in both $\lone$ and $\cstar$. It was observed in \cite{svensson_tomiyama} that its commutant $\cstarcomm$ in $\cstar$ is commutative again, hence a maximal abelian subalgebra. Moreover, $\cstarcomm$ always has non-zero intersection with every non-zero closed ideal\footnote{In fact: with every non-zero ideal, not necessarily closed or self-adjoint.} in $\cstar$, regardless of the dynamics. This rather non-obvious fact provides a new angle on $\cstar$, and in \cite{svensson_tomiyama} this property was put to good use in the study of this $C^*$-algebra. Likewise, the commutant $\lonecomm$ of $\coeffalg$ in $\lone$ is also a maximal abelian subalgebra, and in \cite{de_jeu_svensson_tomiyama} it was established with some effort that, analogously, $\lonecomm$ has non-zero intersection with every non-zero closed ideal of $\lone$. As for $\cstar$, this proved to be a powerful tool to obtain further structural results about $\lone$. In fact, once this basic property for each of $\cstarcomm$ and $\lonecomm$ has been settled, proofs of further results for $\cstar$ and $\lone$ can be given that are rather similar. The proofs of this property for $\cstarcomm$ and $\lonecomm$, however, are not similar.

In view of the pivotal role of $\lonecomm$ and $\cstarcomm$, it is desirable to have more information on these algebras available, and hence they are the objects of study in the present paper.  We explicitly determine the maximal ideal space of $\lonecomm$, show that is a topological quotient of $\topspace\times\T$, and establish that it coincides with the pure state space of $\lonecomm$. As a consequence, $\lonecomm$ is a hermitian Banach algebra, and it is also shown to be semisimple. Its enveloping $C^*$-algebra turns out to be $\cstarcomm$. With that material in place, it is then possible to investigate projections from $\cstar$ onto $\cstarcomm$ and from $\lone$ onto $\lonecomm$, and we obtain a simple topological dynamical condition which is equivalent to the existence of such projections. We show, for both algebras, that positive projections are unique, if existing, and give an explicit formula, cf.\ Theorems~\ref{t:cstar_projection_unique} and~\ref{t:lone_projection_unique}. For $\cstar$, norm one projections are necessarily positive, so that norm one projections are then also unique. It is conceivable that the ``actual" condition for uniqueness is positivity, and that the norm one case can be covered for $\cstar$ only coincidentally, because the projection theorem for $C^*$-algebras happens to show that norm one implies positivity. At the moment this is not clear, and the uniqueness issue for norm one projections from $\lone$ onto $\lonecomm$ needs further study.

In spite of the progress made in \cite{de_jeu_svensson_tomiyama} and the present paper, there are still several intriguing natural questions open about the involutive algebra $\lone$ itself, and about the relation with its enveloping $C^*$-algebra.  For example, is $\lone$ always hermitian? If $\topspace$ consists of finitely many points, then this is the case (cf.\ \cite{de_jeu_svensson_tomiyama}). Furthermore, if $\lone$ is hermitian, then by \cite[Theorem~9.8.2]{palmer_two} every maximal abelian $\sp{*}$-subalgebra is also hermitian. Hence the fact that, from the present paper, the maximal abelian $\sp*$-subalgebra $\lonecomm$ is known to be hermitian for general $X$ can be regarded as additional support for a possible hypothesis that $\lone$ is always hermitian, but presently the general answer is still unknown. Furthermore, if $I$ is a proper closed ideal of $\lone$, is the closure of $I$ in $\cstar$ then proper again? If so, this would enable us to translate results between $\lone$ and $\cstar$, but at this moment the answer is known (and affirmative) only when $\topspace$ consists of one point. We hope that the results in this paper are a further step towards answering such questions about these involutive Banach algebras, of which $\ell^1$ is the easiest example, and which are the involutive Banach algebras of crossed product type most naturally associated with a single homeomorphism of a compact Hausdorff space.

\medskip
This paper is organised as follows.

Section~\ref{s:preliminaries} contains the basic definitions, notations, and preliminary results.

Section~\ref{s:maxidealspace_lone} is concerned with $\lonecomm$, the commutant of $\coeffalg$ in $\lone$. We start by studying its maximal ideal space $\maxidealspacelonecomm$. All characters can be described explicitly, and it follows by inspection that they are all hermitian, so that $\lonecomm$ is a hermitian algebra. Given the description of $\maxidealspacelonecomm$ it is then straightforward to show that $\lonecomm$ is semisimple, and that $\maxidealspacelonecomm$ is a topological quotient of $\topspace\times\T$. The latter fact can be used to generate dense subsets of $\maxidealspacelonecomm$, which will then separate the points of $\lonecomm$ as a consequence of the semisimplicity.

In Section~\ref{s:extending_and_restricting}, $\cstarcomm$ is brought into play, and we study how pure states extend and restrict in the chain $\coeffalg\hookrightarrow\lonecomm\hookrightarrow\cstarcomm\hookrightarrow\cstar$ of unital involutive Banach algebras. The maximal ideal space $\maxidealspacecstarcomm$ of $\cstarcomm$ can then be shown to be homeomorphic with $\maxidealspacelonecomm$.

Section~\ref{s:enveloping_algebra_and_projections} contains the applications of the results in Sections~\ref{s:maxidealspace_lone} and~\ref{s:extending_and_restricting}. It is shown that $\cstarcomm$ is the enveloping $C^*$-algebra of $\lonecomm$, which, in retrospect, explains the homeomorphism between $\maxidealspacelonecomm$ and $\maxidealspacecstarcomm$. The rest of the section consists of a detailed study of projections from $\cstar$ onto $\cstarcomm$, and from $\lone$ onto $\lonecomm$.

In the Appendix we have included, all in one place, topological results on the periodic points that have proved to be useful in various papers. Our results, valid for a homeomorphism of a locally compact Hausdorff space, are actually somewhat stronger than in the existing literature and have a common basis in a new general result on equal closures.

\section{Preliminaries}\label{s:preliminaries}

Throughout this paper, with the exception of the Appendix, $\topspace$ is a non-empty compact Hausdorff space, and $\homeo:\topspace\to\topspace$ is a homeomorphism. We let $\aperpoints$ and $\perpoints$ denote the aperiodic and the periodic points of $\homeo$, respectively. For $n\in\Z$, let $\fixn$ be the set of points fixed by $\homeo^n$, and, for $p\geq 1$, let $\perpnew$ be the set of points of period $p$. If $S\subset\topspace$, then we will denote its interior and closure by $S^\intsymbol$ and $\overline{S}$, respectively. For typographical reasons, we will write $\fixnint$ and $\pernint$, rather than $\fix_n(\homeo)^\intsymbol$ and $\per_n(\homeo)^\intsymbol$.

We let $\coeffalg$ denote the algebra of continuous complex-valued functions on $\topspace$, and write $\alpha$ for be the automorphism of $\coeffalg$ induced by $\homeo$ via $\alpha(f) := f \circ \homeo^{-1}$, for $f \in\coeffalg$. Via $n \mapsto \alpha^n$, the integers act on $\coeffalg$. With $\|\cdot\|_{\infty}$ denoting the supremum norm on $\coeffalg$, we let
\[
\lone = \{\ell: \mathbb{Z} \to\coeffalg : \Vert \ell\Vert := \sum_{k} \|\ell(k)\|_{\infty} < \infty\}.
\]
We supply $\lone$ with the usual multiplication and involution:
\[
(\ell\ell^\prime) (n) := \sum_{k \in \mathbb{Z}} \ell(k) \cdot \alpha^k (\ell^\prime(n-k))\quad(\ell, \ell^\prime \in \lone),
\]
and
\[
\ell^* (n) = \overline{\alpha^n (\ell(-n))} (\ell\in\lone),
\]
so that it becomes a unital Banach $\sp\ast$-algebra with isometric involution.

A convenient way to work with $\lone$ is provided by the following. For $n,m \in \mathbb{Z}$, let
\begin{equation*}
  \chi_{\{n\}} (m) =
  \begin{cases}
    1 &\text{if }m =n;\\
    0 &\text{if }m \neq n,
  \end{cases}
\end{equation*}
where the constants denote the corresponding constant functions in $\coeffalg$. Then $\chi_{\{0\}}$ is the identity element of $\lone$. Let $\delta = \chi_{\{1\}}$; then $\chi_{\{-1\}}=\delta^{-1}=\delta^*$. If we put $\delta^0=\chi_{\{0\}}$, then $\delta^n = \chi_{\{n\}}$, for all $n \in \mathbb{Z}$. We may view $\coeffalg$ as a closed abelian $^*$-subalgebra of $\lone$, namely as $\{f_0 \delta^0 \, : \, f_0 \in \coeffalg\}$. If $\ell \in \lone$, and if we write $\ell(k)=f_k$ for short, then $\ell= \sum_{k} f_k\delta^k$, and $\Vert \ell\Vert=\sum_k \Vert f_k\Vert_\infty<\infty$.
In the rest of this paper we will constantly use this series representations $\ell= \sum_{k}f_k\delta^k$ of an arbitrary element $\ell\in\lone$, for uniquely determined $f_k\in\coeffalg$. Thus $\lone$ is generated as a unital Banach algebra by an isometrically isomorphic copy of $\coeffalg$ and the elements $\delta$ and $\delta^{-1}$, subject to the relation $\delta f \delta^{-1}=f\circ\homeo^{-1}$, for $f \in\coeffalg$. The isometric involution is determined by $f^*=\overline f$, for $f\in\coeffalg$, and $\delta^*=\delta^{-1}$.

The crossed product $C^*$-algebra $\cstar$ is the enveloping $C^*$-enveloping algebra of $\lone$; we refer to \cite{williams} for the general theory of such algebras. According to \cite[2.7.1]{dixmier}, the enveloping $C^*$-seminorm of an element $\ell\in\lone$ can be calculated as
\[
\Vert\ell\Vert^\prime=\sup_{\st\in B} \st(\ell^*\ell)^{1/2},
\]
where $B$ is the set of continuous positive forms on $\lone$ of norm at most 1.\footnote{Since $\lone$ is unital, a positive form $\st$ on $\lone$ is automatically continuous, and $\Vert\st\Vert=\st(1)$, cf.\ \cite[2.1.4]{dixmier}.} For $\ptx\in\topspace$, $\st_{\ptx}(a):=f_0(x)$, for $\ell=\sum_k f_k\delta^k\in\lone$, defines a continuous positive form of norm 1, and since $\st_x(\ell^*\ell)=\sum_k{\vert f_k(\homeo^kx)\vert^2}$, one sees that $\Vert\,.\,\Vert^\prime$ is actually a norm on $\lone$. Thus we can view $\lone$ as a dense subalgebra of $\cstar$.

It can be shown \cite{tomiyama_book, tomiyama_notes_one} that the canonical norm one projection $\Eone:\lone\to\coeffalg$, defined by $\Eone(\ell)=f_0$, if $\ell=\sum_k f_k\delta^k\in\lone$, extends to a faithful norm one projection $\Estar:\cstar\to\coeffalg$, (automatically) such that $\Estar(fcg)=f\Estar(c)g$, for $c\in\cstar$, and $f,g\in\coeffalg$.\footnote{In a purely $C^*$-algebra context it is customary to write $E$ for this projection in $\cstar$. With $\lone$ also under consideration, $\Estar$, together with $\Eone$, is now a more appropriate notation.} If $c\in\cstar$, and $k\in\Z$, then we define its generalised Fourier coefficient $c(k)\in\coeffalg$ as $c(k):=\Estar(c\cdot\delta^{-k})$. Then
\begin{equation}\label{e:projection_properties}
(f\cdot c)(k)=f\cdot c(k),\quad (c\cdot f)(k)=(f\circ\homeo^{-k})\cdot c(k),
\end{equation}
for $f\in\coeffalg$, $c\in\cstar$, and $n\in\Z$. Furthermore, if $c\in\cstar$, then its sequence of Ces\`aro means $(\cesaro_N(c))_N$, defined by
\[
\cesaro_N(c):=\sum_{j=-N}^N \left(1-\frac{|j|}{N+1}\right)c(j)\delta^j,
\]
converges to $c$ in the norm of $\cstar$ as $N\to\infty$, cf.\ \cite[Theorem~1]{tomiyama_structure}.

We let $\lonecomm$ denote the commutant of $\coeffalg$ in $\lone$, and $\cstarcomm$ denote the commutant of $\coeffalg$ in $\cstar$. It is not too difficult to see, cf.\ \cite[Propositions~3.1 and 3.2]{de_jeu_svensson_tomiyama}, that
\begin{equation}\label{e:lonecommutant_description}
\lonecomm=\left\{\sum_k f_k\delta^k\in\lone : \supp(f_k)\subset\fixk\textup{ for all }k\in\Z\right\},
\end{equation}
and that $\lonecomm$ is commutative again, hence an involutive maximal abelian subalgebra of $\lone$. Likewise, cf.\ \cite[Proposition~3.2]{svensson_tomiyama},
\begin{equation}\label{e:cstarcommutant_description}
\cstarcomm=\left\{c\in\cstar : \supp(c(k))\subset\fixk\textup{ for all }k\in\Z\right\},
\end{equation}
and $\cstarcomm$ is again commutative, hence an involutive maximal abelian subalgebra of $\cstar$.

We let $\fs$ denote the elements of $\lone$ of the form $\ell=\sum_{k} f_k \delta^k$, where the summation is finite. It is a unital involutive subalgebra of $\lone$, which is dense in $\lone$ in the norm of $\lone$, and hence dense in $\cstar$. Let $\fscomm$ denote the commutant of $\coeffalg$ in $\fs$. Clearly $\fscomm=\lonecomm\cap\fs=\cstarcomm\cap\fs$. It is obvious from \eqref{e:lonecommutant_description} that $\fscomm$ is dense in $\lonecomm$ in the norm of $\lone$. It is also true that $\fscomm$ is dense in $\cstarcomm$. Indeed, if $c\in\cstarcomm$, then it is clear from \eqref{e:cstarcommutant_description} that each of the Ces\`aro means $\cesaro_N(c)$ is an element of $\fscomm$, and we know these means to converge to $c$. The density of $\fscomm$ in $\cstarcomm$ implies that $\lonecomm$ is dense in $\cstarcomm$. To summarise, we have the following inclusions:

\begin{equation*}
\xymatrix
{
&&&\lone\ar @{^{(}->}[dd]^{\text{dense}}\\
\coeffalg\ar @{^{(}->}[r]&\fscomm\ar @{^{(}->}[r]&\fs\quad\ar @{^{(}->}[ru]^{\text{dense}}\ar @{^{(}->}[rd]^{\text{dense}}&\\
&&&\cstar\\
&&\lonecomm\ar @{^{(}->}[r]\ar @{^{(}->}[dd]^{\text{dense}}&\lone\ar @{^{(}->}[dd]^{\text{dense}}\\
\coeffalg\ar @{^{(}->}[r]&\fscomm\quad\ar @{^{(}->}[ur]^{\text{dense}}\ar @{^{(}->}[dr]^{\text{dense}}&&\\
&&\cstarcomm\ar @{^{(}->}[r]&\cstar
}
\end{equation*}

Furthermore, let us recall that the GNS-theory is available for unital Banach algebras with an isometric involution, cf.\ \cite{dixmier}, so that the pure states of $\lonecomm$ are in bijection with its topologically irreducible representations. Since $\lonecomm$ is commutative, Schur's lemma then implies that the pure states of $\lonecomm$ are precisely the hermitian (i.e., self-adjoint) characters of $\lonecomm$. We will use these terms interchangeably.

\medskip

For convenience, we list explicitly the following elementary facts, to be used without further mention in the sequel:
\begin{enumerate}[\upshape (i)]
\item $\homeo^m\left(\fixn\right)=\fixn$, for all $m,n\in\Z$.
\item If $f\in\coeffalg$, and $\supp (f)\subset\fixn$, for some $n\in\Z$, then $\supp(f\circ\homeo^m)\subset\fixn$, for all $m\in\Z$.
\item If $f\in\coeffalg$, and $\supp (f)\subset\fixn$, for some $n\in\Z$, then $f\circ\homeo^{mn}=f$, for all $m\in\mathbb Z$.
\item $\fixm\cap\fixn=\fix_{\gcd(m,n)}(\homeo)$, for all $m,n\geq 0$.
\end{enumerate}

We conclude our preliminaries with the following result, which is also elementary, but nevertheless central to several arguments.

\begin{lem}\label{l:one}
 Suppose $\pt\in\fixnint$, for some $n\geq 1$, and that $n$ is minimal with these properties. If $f\in\coeffalg$, $m\in\Z$, $\supp (f)\subset\fixm$, and $n\nmid m$, then $f(\pt)=0$.
\end{lem}

\begin{proof}
We may assume that $m\geq 1$. There exists an open neighbourhood $U$ of $\pt$, such that $U\subset\fixn$. If $f(\pt)\neq 0$, then there exists an open neighbourhood $V$ of $\pt$, such that $V\subset\fixm$. Then $\pt\in U\cap V\subset \fixn\cap\fixm=\fix_{\gcd (n,m)}(\homeo)$, so that $\pt\in\fix_{\gcd (n,m)}^\intsymbol(\homeo)$. Now $1\leq\gcd(n,m)$, and $\gcd(n,m)<n$, since $n\nmid m$. Hence the minimality of $n$ is contradicted.
\end{proof}

\section{The commutant $\lonecomm$ of $\coeffalg$ in $\lone$}\label{s:maxidealspace_lone}

In this section we study the $\lonecomm$, the commutant of $\coeffalg$ in $\lone$. We describe its maximal ideal space $\maxidealspacelonecomm$ as a set in Proposition~\ref{p:nine}, and as a topological quotient of $\topspace\times\T$ in Theorem~\ref{t:topological_quotient}. It turns out, cf.\ Proposition~\ref{p:nine}, that all characters of $\lonecomm$ are hermitian, or, equivalently, that $\maxidealspacelonecomm$ coincides with the pure state space of $\lonecomm$. As a consequence, $\lonecomm$ is a hermitian involutive Banach algebra and, moreover, it can be shown to be semisimple, cf.\ Theorem~\ref{t:hermitian_and_semisimple}. It is then easy to generate dense subsets of $\maxidealspacelonecomm$, which will separate the points of $\lonecomm$, as in Corollary~\ref{c:separating_lonecomm}.

The first step is to determine all characters of the algebra $\fscomm$, i.e., all non-zero complex homomorphisms, including those which are not continuous in the norm induced from $\lone$. Continuity can then easily be taken into account later on. The next lemma describes the basic structure of possible characters of $\fscomm$, showing that they are closely related to point evaluations $\ev_{\pt}$ on $\coeffalg$.

\begin{lem}\label{l:two}
Let $\ch$ be a character of $\fscomm$, and let $\pt\in \topspace$ be the unique point such that $\ch\rest_\coeffalg=\ev_{\pt}$.
Then there exists a sequence $\left( c_k\right )_{k\in\Z}$, with $c_0=1$, such that $\ch(f\delta^k)=c_k f(\pt)$, for all $k\in\Z$ and all $f\in\coeffalg$ with $\supp(f)\subset\fixk$.
\end{lem}

\begin{proof}
For $f\in\coeffalg$ and $k\in\Z$ we compute, using that $\supp(f)\subset\fixk$ in the third step:
\begin{align*}
\ch(f\delta^k)^2&=\ch(f\delta^k\cdot f\delta^k)\\
&=\ch (f\cdot[f\circ \homeo^{-k}]\delta^{2k})\\
&=\ch (f^2\delta^{2k})\\
&=\ch(f)\,\ch(f\delta^{2k})\\
&=f(\pt)\ch(f\delta^{2k}).
\end{align*}
Hence, if we let $V_k:=\{f\in\coeffalg : \supp(f)\subset\fixk\}$, then the linear map $\ch_k:V\to\C$, defined by $f\mapsto\ch(f\delta^k)$, for $f\in V_k$, has a kernel containing the kernel of $\ev_{\pt}\rest_{V_k}$. Therefore, there exists $c_k\in\C$, such that $\ch_k=c_k\ev_{\pt}\rest_{V_k}$, as claimed. Clearly $c_0=1$.
\end{proof}

 We will now investigate this further, depending on the properties of $\pt$. The relevant criterion will turn out to be whether $\pt$ is an element of $\intfixedpoints$ or not.

\begin{lem}\label{l:three}
Let $\ch$ be a character of $\fscomm$, and let $\pt\in \topspace$ be the unique point such that $\ch\rest_\coeffalg=\ev_{\pt}$. Suppose that $\pt\notin\intfixedpoints$. Then $\ch\left(\sum_k f_k\delta^k\right)=f_0(\pt)$, for all $\sum_k f_k\delta^k\in\fscomm$.
\end{lem}

\begin{proof}
From Lemma~\ref{l:two} we have $\ch\left(\sum_k f_k\delta^k\right)=\sum_k c_k f_k(\pt)$, for some sequence $(c_k)_{k\in\Z}$ with $c_0=1$. However, for $k\neq 0$ we must have $f_k(\pt)=0$, since otherwise $\pt\in\fix_{|k|}^\intsymbol(\homeo)$, which would contradict the assumption on $\pt$.
\end{proof}

Lemma~\ref{l:three} shows that, for $\pt\notin\intfixedpoints$, the character $\ev_{\pt}$ of $\coeffalg$ has at most one extension to a character of $\fscomm$. This candidate is actually a character, as is asserted by the following lemma.

\begin{lem}\label{l:four}
If  $\pt\notin\intfixedpoints$, then defining $\ch_{\pt}\left(\sum_k f_k\delta^k\right):=f_0(\pt)$, for $\sum_k f_k\delta^k\in\fscomm$, yields a character $\ch_{\pt}$ of $\fscomm$ extending $\ev_{\pt}$.
\end{lem}

\begin{proof}
We check multiplicativity on a spanning set. Let $n,m\in\Z$, and suppose $f_n,f_m\in\coeffalg$ with $\supp (f_n)\subset\fixn$, and $\supp (f_m)\subset\fixm$. Then we must check that
\begin{equation}\label{e:multiplicativity}
\ch_{\pt}\left(f_n\delta^n\cdot f_m\delta^m\right)=\ch_{\pt}\left(f_n\delta^n\right)\ch_{\pt}\left(f_m\delta^m\right).
\end{equation}
We distinguish two cases. In the first case, if $n=m=0$, then the left hand side and the right hand side in \eqref{e:multiplicativity} are both equal to $f_n(\pt)f_m(\pt)$, and we are done. In the second case, if at least one of $n$ and $m$ is non-zero, then the right hand side in \eqref{e:multiplicativity} is zero. The left hand side is equal to $\ch_{\pt}\left(f_n\cdot[f_m\circ\homeo^{-n}]\delta^{n+m}\right)$. If, as a subcase, $n+m\neq 0$, then this left hand side is equal to zero, and we are done. If, as the remaining subcase, $n+m=0$, then in particular $n=-m\neq 0$, and the left hand side equals $f_n(\pt)f_m(\homeo^{-n}\pt)$. Since $\pt\notin\intfixedpoints$, it follows that $f_n(\pt)=0$, as otherwise $\pt\in\fix_{|n|}^\intsymbol(\homeo)$, contradicting the assumption on $\pt$. Hence we are done with the subcase $n+m=0$ as well.
\end{proof}

Having taken care of the characters of $\fscomm$ extending $\ev_{\pt}$, for $\pt\notin\intfixedpoints$, we turn to the case where $\pt\in\intfixedpoints$. As we will see, the extensions of $\ev_{\pt}$ are then parametrised by $\Ccross:=\C\setminus\{0\}$. As above, we start by first describing potential extensions.

\begin{lem}\label{l:five}
Let $\ch$ be a character of $\fscomm$, and let $\pt\in \topspace$ be the unique point such that $\ch\rest_\coeffalg=\ev_{\pt}$.
 Suppose that $\pt\in\intfixedpoints$. Let $n$ be the minimal $n\geq 1$ such that $\pt\in\fixnint$. Then there exists a unique $c\in\Ccross$, such that $\ch\left(\sum_k f_k\delta^k\right)=\sum_j f_{jn}(\pt)c^j$, for all $\sum_k f_k\delta^k\in\fscomm$.
\end{lem}

\begin{proof}
Lemma~\ref{l:two} provides a sequence $(c_k)_{k\in\Z}$, such that $\ch\left(\sum_k f_k\delta^k\right)=\sum_k c_k f_k (\pt)$. By Lemma~\ref{l:one}, this simplifies to $\ch(\sum_k f_k \delta^k)=\sum_j c_{jn} f_{jn}(\pt)$, so it remains to show that there exists $c\in\Ccross$, evidently unique, such that $c_{jn}=c^j$, for all $j\in\mathbb Z$. Since $\pt\in\fixnint$, there exists $f_0\in\coeffalg$ with $\supp (f_0)\subset\fixnint$ and $f_0(\pt)=1$. Then
\begin{equation}\label{e:lemmafive}
\ch\left(f_0 \delta^n\cdot f_0\delta^{-n}\right)=\ch\left(f_0\delta^n\right)\ch\left(f_0\delta^{-n}\right).
\end{equation}
The right hand side in \eqref{e:lemmafive} equals $c_n f_0(\pt) c_{-n} f_0(\pt)=c_n c_{-n}$.
The left hand side equals $\ch(f_0\cdot[f_0\circ\homeo^{-n}])=\ch(f_0^2)=f_0^2(\pt)=1$. Write $c=c_n$; then $c\neq 0$.
\\For $j\geq 1$ we have $(f_0\delta^n)^j=f_0^j\delta^{jn}$, hence $\ch((f_0\delta^n)^j)=\ch(f_0^j \delta^{jn})$, so $c^j=c_{jn}f_0^j(\pt)=c_{jn}$. Also, for $j\geq 1$, $(f_0\delta^{-n})^j=f_0^j\delta^{-jn}$, hence $\ch((f_0\delta^{-n})^j)=\ch(f_0^j\delta^{-jn})$, so $(c_{-n})^j=c_{-jn} f_0^j(\pt)=c_{-jn}$. Since $c_{-n}=\frac{1}{c_n}=\frac{1}{c}$, we find that $c_{-jn}=c^{-j}$. Hence $c_{jn}=c^j$, for all $j\in\Z$.
\end{proof}

Just as before, the candidate extensions are, in fact, characters of $\fscomm$, as is asserted by the next result.

\begin{lem}\label{l:six}
Let $\pt\in\intfixedpoints$, and suppose $n$ is the minimal $n\geq 1$ such that $\pt\in\fixnint$. Let $c\in\Ccross$. Then $\ch_{\pt,c}(\sum_k f_k\delta^k):=\sum_j f_{jn}(\pt)c^j$, for $\sum_k f_k\delta^k\in\fscomm$, defines a character $\ch_{\pt,c}$ of $\fscomm$ extending $\ev_{\pt}$.
\end{lem}

\begin{proof}
We check multiplicativity on a spanning set. Let $k,l\in\Z$, $f_k,f_l\in\coeffalg$, and suppose $\supp (f_k)\subset\fixk$, and $\supp (f_l)\subset \fixl$. We must check that
\begin{equation}\label{e:lemmasix}
\ch_{\pt,c}\left(f_k\delta^k\cdot f_l\delta^l\right)=\ch_{\pt,c}\left(f_k\delta^k\right)\ch_{\pt,c}\left(f_l\delta^l\right).
\end{equation}
There are three cases to consider.
\\If $n\mid k$ and $n\mid l$, then the right hand side \eqref{e:lemmasix} is equal to $f_k(\pt)c^{k/n}\cdot f_l(\pt)c^{l/n}$. The left hand side in \eqref{e:lemmasix}is equal to $f_k(\pt)f_l(\homeo^{-k}\pt)c^{(k+l)/n}$. Since $\pt\in\fixn$, and $n\mid k$, this equals the right hand side.
\\If $n\nmid k$, then the right hand side in \eqref{e:lemmasix} is zero, since the first factor is. If, as a subcase, $n\nmid(k+l)$, then the left hand side in \eqref{e:lemmasix} is zero, and we are done. If, in the remaining subcase, $n\mid(k+l)$, then the left hand side in \eqref{e:lemmasix} is equal to $f_k(\pt)f_l(\homeo^{-k}\pt)c^{(k+l)/n}$. Since $n\nmid k$, we have $f_k(\pt)=0$ by Lemma~\ref{l:one}, and so the left hand side in \eqref{e:lemmasix} is also seen to be zero.
\\In $n\nmid l$, then \eqref{e:lemmasix} holds as a consequence of the second case and the commutativity of $\lonecomm$.
\end{proof}

To summarise:

\begin{lem}\label{l:seven}
The characters of $\fscomm$ are, without multiple occurrences, the following, where $\sum_k f_k\delta^k\in\fscomm$:
\begin{enumerate}[\upshape (i)]
\item for $\pt\notin \intfixedpoints$:
\[
\ch_{\pt}\left(\sum_k f_k\delta^k\right)=f_0(\pt).
\]
\item for $\pt\in\intfixedpoints$, and $c\in\Ccross$:
\[
\ch_{\pt,c}\left(\sum_k f_k\delta^k\right)=\sum_j f_{jn}(\pt) c^j,
\]
where $n$ is the minimal $n\geq 1$ such that $\pt\in\fixnint$.
\end{enumerate}
\end{lem}

It is easy to see which characters of $\fscomm$ are hermitian, and which are continuous in the norm from $\lone$.

\begin{lem}\label{l:eight}
\quad
\begin{enumerate}[\upshape (i)]
\item For $\pt\notin\intfixedpoints$, $\ch_{\pt}$ is a hermitian character of $\fscomm$, which is continuous in the topology induced by $\lone$.
\item Let $\pt\in\intfixedpoints$, and $c\in\Ccross$. Then the following are equivalent:
\begin{enumerate}[\upshape (i)]
\item $\ch_{\pt,c}$ is continuous on $\fscomm$ in the topology induced by $\lone$.
\item $\ch_{\pt,c}$ is a hermitian character of $\fscomm$.
\item $c\in\T$.
\end{enumerate}
\end{enumerate}
\end{lem}

\begin{proof}
The first part is obvious. As to the second, it is clear that (c) implies (a). To see that (a) implies (c), let $n$ be as in Lemma~\ref{l:seven}. Pick $f_0\in\coeffalg$, such that $\supp (f_0)\subset\fixn$, and $\Vert f_0\Vert_\infty=f_0(\pt)=1$. Then, for all $j\in\Z$, $\Vert f_0\delta^{jn}\Vert=1$, and $f_0\delta^{jn}\in\fscomm$. The assumed continuity implies that $|\ch_{\pt,c}(f_0\delta^{jn})|=|c^j|$ remains bounded, for $j\in\Z$. Hence $c\in\T$. In order to show the equivalence between (b) and (c) we compute, for $\pt\in\intfixedpoints$, $c\in\Ccross$, and $\sum_k f_k\delta^k\in\fscomm$, with $n$ as in Lemma~\ref{l:seven} again:
\begin{align*}
(\ch_{\pt,c})^* \left(\sum_k f_k \delta^k\right)&=\overline{\ch_{\pt,c}\left(\sum_k \delta^{-k}\overline{f}_k \right)}\\
&=\overline{\ch_{\pt,c}\left(\sum_k (\overline{f}_k \circ\homeo^k)\delta^{-k}\right)}\\
&=\overline{\ch_{\pt,c}\left(\sum_k (\overline{f}_{-k}\circ\homeo^{-k})\delta^k \right)}\\
&=\overline{\sum_j(\overline{f}_{-jn} \circ\homeo^{-jn})(\pt)c^j}\\
&=\overline{\sum_j(\overline{f}_{-jn})(\pt)c^j}\\
&=\sum_j f_{jn}(\pt)(1/\overline{c})^j \\
&=\ch_{\pt,1/\overline{c}}\left(\sum_k f_k \delta^k\right),
\end{align*}
where it was used in the fifth step that $\pt\in\fixn$. Hence $\ch_{\pt,c}^*=\ch_{\pt,1/\overline{c}}$. From the uniqueness of the parametrisation we see that $\ch_{\pt,c}^*=\ch_{\pt,c}$ if and only if $1/\overline{c}=c$, i.e., if and only if $c\in\T$.
\end{proof}

The characters of $\fscomm$ which extend to characters of $\lonecomm$ are the continuous ones, and from Lemma~\ref{l:seven} and~\ref{l:eight} we know what these are. For such characters, we will employ the same notation when viewing them as a character of $\lonecomm$. Hence we have the following description of the maximal ideal space $\maxidealspacelonecomm$ of $\lonecomm$ as a set. In Theorem~\ref{t:topological_quotient} we will also describe its topology.

\begin{prop}\label{p:nine}
The characters of $\lonecomm$ are, without multiple occurrences, the following, where $\sum_k f_k\delta^k\in\lonecomm$:
\begin{enumerate}[\upshape (i)]
\item for $\pt\notin \intfixedpoints$:
\[
\ch_{\pt}\left(\sum_k f_k\delta^k\right)=f_0(\pt).
\]
\item for $\pt\in\intfixedpoints$, and $c\in\T$:
\[
\ch_{\pt,c}\left(\sum_k f_k\delta^k\right)=\sum_j f_{jn}(\pt)c^j,
\]
where $n$ is the minimal $n\geq 1$ such that $\pt\in\fixnint$.
\end{enumerate}
All characters of $\lonecomm$ are hermitian, and the map from $\maxidealspacelonecomm$ into $\Delta(\coeffalg)$ given by restriction of characters is a continuous surjection.
\end{prop}

As a consequence of Proposition~\ref{p:nine} we have the following.

\begin{thm}\label{t:hermitian_and_semisimple}
The Banach algebra $\lonecomm$ is hermitian and semisimple.
\end{thm}

\begin{proof}
It is clear from Proposition~\ref{p:nine} that the spectrum of a self-adjoint element of $\lonecomm$ is real, i.e., that $\lonecomm$ is hermitian.  To show that it is semisimple, suppose $\sum_k f_k\delta^k\in\lonecomm$ and that $\ch(\sum_{k} f_k\delta^k)=0$, for all $\ch\in\maxidealspacelonecomm$. In that case, if $\pt\notin \intfixedpoints$, then $f_0(\pt)=\ch_{\pt}(\sum_k f_k\delta^k)=0$. If $n\neq 0$, then also $f_n(\pt)=0$, because otherwise $\pt\in\fix_{|n|}^\intsymbol(\homeo)$. Hence, for all $k\in\Z$, $f_k$ vanishes at the complement of $\intfixedpoints$. In order so show that all $f_k$ also vanish at an arbitrary $\pt\in\intfixedpoints$, let $n$ be the minimal $n\geq 1$ such that $\pt\in\fixnint$. Then, by Lemma~\ref{l:one}, we know that $f_m(\pt)=0$ if $n\nmid m$. To show that this is likewise true if $m=jn$, for arbitrary $j\in\Z$, we use that, for all $c\in\mathbb T$, $0=\ch_{\pt,c}(\sum_k f_k\delta^k)=\sum_j f_{jn}(\pt) c^j$. That is: the Fourier transform of the element $(\ldots,f_{-2n}(\pt),f_{-n}(\pt),f_0(\pt),f_n(\pt), f_{2n}(\pt),\ldots)$ of $\ell^1(\Z)$ is equal to zero. We conclude that $f_{jn}(\pt)=0$, for all $j\in\Z$, as was to be shown.
\
\end{proof}

We will now proceed to show that $\maxidealspacelonecomm$ is a topological quotient of $\topspace\times\T$. If $\pt\in\topspace$, and $z\in\T$, define $\surjmap_{\pt,z}:\lone\to\C$ by $\surjmap_{\pt,z}(\sum_k f_k \delta^k):=\sum_k f_k(\pt)z^k$, for $\sum_k f_k\delta^k\in\lonecomm$. Since $z\in\T$, this is well-defined. We claim that $\surjmap_{\pt,z}$ is, in fact, a character of $\lonecomm$. There are two cases two consider:
\begin{enumerate}[\upshape (i)]
\item If $\pt\notin\intfixedpoints$, then $f_k(\pt)=0$, for all $k\neq 0$. Hence, regardless of $z$, $\surjmap_{\pt,z}=\ch_{\pt}$ in this case, which we know to be a character of $\lonecomm$.
\item If $\pt\in\intfixedpoints$, we let $n$ be the minimal $n\geq 1$ such that $\pt\in\fixnint$. Then, by Lemma~\ref{l:one}, $f_m(\pt)=0$ if $n\nmid m$, so $\surjmap_{\pt,z}(\sum_k f_k\delta^k)=\sum_j f_{jn}(\pt)z^{jn}=\ch_{\pt,z^n}(\sum_k f_k \delta^k)$. We conclude that $\surjmap_{\pt,z}=\ch_{\pt,z^n}$ which, again, we already know to be a character of $\lonecomm$.
\end{enumerate}
We conclude that $\surjmap$, sending $(\pt,z)$ to $\surjmap_{\pt,z}$, does not only map $\topspace\times\T$ into $\maxidealspacelonecomm$, but also that it is surjective. In fact, we have the following.

\begin{thm}\label{t:topological_quotient}
The maximal ideal space $\maxidealspacelonecomm$ is the topological quotient of $\topspace\times\T$ via the surjective continuous map $\surjmap:\topspace\times\T\to\maxidealspacelonecomm$. The fibers are:
\begin{enumerate}[\upshape (i)]
\item for $\ch_{\pt}$, with $\pt\notin \intfixedpoints$: $\surjmap^{-1}(\{\ch_{\pt}\})=\{\pt\}\times\T$.
\item for $\ch_{\pt,c}$, with $\pt\in\intfixedpoints$, and $c\in\T$: $\surjmap^{-1}(\{\ch_{\pt,c}\})=\{(\pt,z) : z^n=c\}$, where $n$ is the minimal $n\geq 1$ such that $\pt\in\fixnint$.
\end{enumerate}
\end{thm}

\begin{proof}
All is clear from the above discussion and Proposition~\ref{p:nine}, except for the statement on the topological quotient. Since any quotient of $\topspace\times\T$ is compact, and $\maxidealspacelonecomm$ is Hausdorff, it is sufficient to show that $\surjmap$ is continuous. To this end, fix $(\pt,z_0)\in\topspace\times\T$ and $\sum_{k\in\Z} f_k\delta^k\in\lonecomm$, and let $\epsilon>0$ be given. There exists $N\geq 0$, such that $\Vert \sum_{k\in\Z} f_k\delta^k-\sum_{k=-N}^{N} f_k\delta^k)\Vert<\epsilon/4$.
Then, if $(\ptx,z)\in\topspace\times\T$, using the contractivity of characters of a Banach algebra, we find that
\begin{align*}
|\surjmap_{\pt,z_0}(\sum_{k\in\Z} f_k \delta^k)-\surjmap_{x,z}(\sum_{k\in\Z} f_k\delta^k)|&<\frac{\epsilon}{2} + |\surjmap_{\pt,z_0}(\sum_{k=-N}^{N} f_k \delta^k)-\surjmap_{x,z}(\sum_{k=-N}^{N} f_k\delta^k)|\\
&\leq \frac{\epsilon}{2}+ \sum_{k=-N}^{N}| f_k(\pt)z_0^k- f_k(\ptx)z^k|.
\end{align*}
It is clear that there exist open neighbourhoods $U_{\pt}$ of $\pt$, and $V_{z_0}$ of $z_0$, such that the finite summation in the above equation is less than $\epsilon/2$, for all $(\ptx,z)\in U_{\pt}\times V_{\pt}$.
\end{proof}

The continuity and surjectivity of $\surjmap$ in Theorem~\ref{t:topological_quotient} provide a convenient tool to generate dense subsets of $\maxidealspacelonecomm$. Indeed, $\psi(\tilde \topspace\times \tilde\T)$ will be such a set, whenever $\tilde \topspace$ is dense in $\topspace$, and $\tilde\T$ is dense  in $\T$. Moreover, since $\lonecomm$ is semisimple, $\psi(\tilde \topspace\times \tilde T)$ will then also separate the points of $\lonecomm$. Applying this with $\tilde\topspace$ equal to the set in the third part of Corollary~\ref{c:density_in_topspace}, and with $\tilde T=\T$, we obtain the following result. It will be instrumental in the proof of the uniqueness statement for projections in $\lone$ in Theorem~\ref{t:lone_projection_unique} and, when combined with the homeomorphism of maximal ideal spaces in Proposition~\ref{p:homeomorphism_of_max_ideal_spaces}, also for the uniqueness statement for projections in $\cstar$ in Theorem~\ref{t:cstar_projection_unique}.

\begin{cor}\label{c:separating_lonecomm}
The set, consisting of all $\ch_{\pt}$, for $\pt\in\aperpoints$, together with all $\ch_{\pt,c}$, for $\pt\in\intperpoints$ and $c\in\T$, is a dense subset of $\maxidealspacelonecomm$, which separates the points of $\lonecomm$.
\end{cor}

\section{Extending and restricting pure states}\label{s:extending_and_restricting}

In this section, we study the behaviour of pure states under restriction in the following chain of involutive algebras:
\begin{equation}\label{e:inclusions}
\xymatrix@1
{
\coeffalg\,\,\ar @{^{(}->}[r]&\,\,\lonecomm\,\,\ar @{^{(}->}[r]^{\text{dense}}&\,\,\cstarcomm\,\,\ar  @{^{(}->}[r]&\,\,\cstar.
}
\end{equation}
The results are summarised in Theorem~\ref{t:summarising_theorem} and will be applied in Section~\ref{s:enveloping_algebra_and_projections}.

As a preliminary remark, we note that $\lonecomm$ need not be a $C^*$-algebra, hence the standard extension theorem for (pure) states does not apply if this algebra is involved. This complicates the arguments, but in the end the the picture for the first two inclusions in \eqref{e:inclusions} turns out to be quite simple.

Beginning with the inclusion $\coeffalg\hookrightarrow\lonecomm$, we recall from Proposition~\ref{p:nine} that all characters of $\lonecomm$ are hermitian, i.e., that the maximal ideal space $\maxidealspacelonecomm$ of $\lonecomm$ coincides with its pure state space. Furthermore, Proposition~\ref{p:nine} also shows that restriction not only maps the pure state space $\maxidealspacelonecomm$ into the pure state space $\Delta(\coeffalg)$, which is obvious since a hermitian character restricts to a hermitian character, but actually onto $\Delta(\coeffalg)$. Moreover, it describes explicitly the fibers of this continuous surjection. Hence the situation for the inclusion $\coeffalg\hookrightarrow\lonecomm$ is already clear.

For the inclusion $\lonecomm\hookrightarrow\cstarcomm$ it is likewise obvious that restriction yields an injective continuous map from the pure state space $\maxidealspacecstarcomm$ into the pure state space $\maxidealspacelonecomm$, but it is not obvious that this map should be surjective. Nevertheless, restriction even yields a homeomorphism between these two pure state spaces (cf.\ Proposition~\ref{p:homeomorphism_of_max_ideal_spaces}), and this will later (cf.\ Theorem~\ref{t:enveloping_algebra}) be improved by showing that $\cstarcomm$ is actually the enveloping $C^*$-algebra of $\lonecomm$.

The key to understanding the restriction map for the inclusion $\lonecomm\hookrightarrow\cstarcomm$ lies in understanding part of the restriction map for the inclusion $\cstarcomm\hookrightarrow\cstar$. To be precise: for $\pt\in\topspace$, we will be concerned with the pure state extensions of the pure state $\ev_{\pt}$ on $\coeffalg$ to $\cstar$, and determine how these extensions restrict to $\cstarcomm$. Although it is not a priori evident, they restrict to pure states. On the other hand, it is a priori evident that all pure states of $\cstarcomm$ can be obtained in this way. Indeed, a pure state $\ch$ of $\cstarcomm$ is a character, so that its restriction to $\coeffalg$ is of the form $\ev_{\pt}$, for some $\pt$ in $\topspace$. Then each pure state extension of $\ch$ to $\cstar$ is also a pure state extension of $\ev_{\pt}$, showing that $\ch$ can be obtained as described.

\medskip
We will now supply the detailed arguments needed to understand the behaviour of pure states for the inclusion $\lonecomm\hookrightarrow\cstarcomm$.  They are based on an explicit description, for all $\pt\in \topspace$, of the restriction of each pure state extension of $\ev_{\pt}$ to $\cstar$ to the dense subalgebra $\fscomm$ of $\cstarcomm$. We start with a description of these pure state extensions to $\cstar$ and the corresponding GNS-representations, referring to \cite[\S 4]{tomiyama_notes_one} for further details and proofs.

First of all, if $\pt\in\aperpoints$, then there is a unique pure state extension of $\ev_{\pt}$ to $\cstar$, which we denote by $\st_{\pt}$. The Hilbert space for the GNS-representation $\pi_{\pt}$ corresponding to $\st_{\pt}$ has an orthonormal basis $(e_n)_{n\in\Z}$, and the representation itself is determined by $\pi_{\pt}(\delta)e_n=e_{n+1}$, for $n\in\Z$, and $\pi_{\pt}(f)e_n=f(\homeo^n \pt)e_n$, for $n\in\Z$. The vector $e_0$ reproduces the state $\st_{\pt}$ of $\cstar$.

If $\pt\in\perpoints$, say $\pt\in\perpnew$ $(p\geq 1)$, then the pure state extensions of $\ev_{\pt}$ to $\cstar$ are in bijection with the points in $\T$, and we denote these pure states of $\cstar$ by $\st_{\pt,\lambda}$, for $\lambda\in\T$. The Hilbert space for the GNS-representation $\pi_{\pt,\lambda}$ corresponding to $\st_{\pt,\lambda}$ has an orthonormal basis $\{e_0,\ldots,e_{p-1}\}$, $\pi_{\pt,\lambda}(\delta)$ is represented with respect to this basis by the matrix
\[
\left(
\begin{array}{ccccc}
0 & 0 & \ldots & 0 & \lambda \\
1 & 0 & \ldots & 0 & 0 \\
0 & 1 & \ldots & 0&0\\
\vdots & \vdots & \ddots &\vdots &\vdots \\
0 & 0 & \ldots & 1&0
\end{array}\right),
\]
and, for $f\in\coeffalg$, $\pi_{\pt,\lambda}(f)$ is represented with respect to this basis by the matrix
\[
\left(
\begin{array}{cccc}
f(\pt) & 0 & \ldots & 0 \\
0 & f (\homeo \pt) & \ldots & 0 \\
\vdots & \vdots & \ddots & \vdots \\
0 & 0 & \ldots & f(\homeo^{p-1} \pt)
\end{array} \right).
\]
The vector $e_0$ reproduces the state $\st_{\pt,\lambda}$ of $\cstar$.

Next, we describe the restriction of these pure states $\st_{\pt}$, for $\pt\in\aperpoints$, and of $\st_{\pt,\lambda}$, for $\pt\in\perpoints$ and $\lambda\in\T$, to $\fscomm$. By density, this will then enable us to obtain information about the restriction to $\cstarcomm$. There are three cases to consider.

\medskip
\noindent\emph{Case 1:} $\pt\in\aperpoints$.
\\ It is immediate from the GNS-model that $\st_{\pt}(f_k\delta^k)=0$, for all $f_k\in\coeffalg$, and all $k\neq 0$. Since $\st_{\pt}(f)=f(\pt)$, for $f\in\coeffalg$, we see that $\st_{\pt}(\sum_k f_k\delta^k)=f_0(\pt)=\ch_{\pt}(\sum_k f_k\delta^k)$, for all $\sum_k f_k\delta^k\in\fscomm=\fscomm$, where $\ch_{\pt}$ is as in the first part of Lemma~\ref{l:seven}. In particular, $\st_{\pt}$ is multiplicative on $\fscomm$. By continuity we conclude that the restriction of $\st_{\pt}$ to $\cstarcomm$ is likewise a character, hence a pure state of $\cstarcomm$.

\medskip
\noindent\emph{Case 2:} $\pt\in\perpoints$, with $\pt\in\intfixedpoints$; $\lambda\in\T$ arbitrary.
\\Let $n$ be the minimal $n\geq 1$ such that $\pt\in\fixnint$, and let $p$ be the period of $\pt$, so that $p\mid n$. As a preparation, we observe that the GNS-model implies that $\st_{\pt,\lambda}(f_k\delta^k)=0$, for all $f_k\in\coeffalg$, and all $k\in\Z$ with $p\nmid k$. Furthermore, if $p\mid k$, then $\st_{\pt,\lambda}(f_k\delta^k)=f(\pt)\lambda^{\frac{k}{p}}$. Turning to $\cstarcomm$, suppose $k\in\Z$, $f_k\in\coeffalg$, and $f_k\delta^k\in\fscomm$, so that $\supp(f_k)\subset\fixk$. According to our preparations, $\st_{\pt,\lambda}(f_k\delta^k)=0$ if $p\nmid k$, and $\st_{\pt,\lambda}(f_k\delta^k)=f_k(\pt)\lambda^{\frac{k}{p}}$ if $p\mid k$. But, since $\supp(f_k)\subset\fixk$, we know in addition from Lemma~\ref{l:one} that $f_k(\pt)=0$ if $n\nmid k$. Hence
\[
\st_{\pt,\lambda}(f_k\delta^k)=
\begin{cases}
f_k(\pt)\lambda^{\frac{k}{p}}&\textup{if }n\mid k,\\
0&\textup{otherwise}.
\end{cases}
\]
Therefore, if $\sum_k f_k\delta^k\in\fscomm$, then
\begin{align*}
\st_{\pt,\lambda}\left(\sum_k f_k \delta^k\right)&=\st_{\pt,\lambda}\left(\sum_j f_{jn}\delta^{jn}\right)\\
&=\sum_j f_{jn}(\pt)\lambda^{\frac{jn}{p}}\\
&=\ch_{\pt,\lambda^{\frac{n}{p}}}\left(\sum_k f_k\delta^k\right),
\end{align*}
where, for $c\in\Ccross$, $\ch_{\pt,c}$ is as in the second part of Lemma~\ref{l:seven}.
As in the first case, since we know $\ch_{\pt,\lambda^{\frac{n}{p}}}$ to be multiplicative on $\fscomm$, we conclude that the restriction of $\st_{\pt,\lambda}$ to $\cstarcomm$ is, in fact, a pure state of $\cstarcomm$.

\medskip
\noindent\emph{Case 3:} $\pt\in\perpoints$, with $\pt\notin\intfixedpoints$; $\lambda\in\T$ arbitrary.
\\ If $p\geq 1$ is the period of $\pt$, then, as in the second case, the GNS-model shows that $\st_{\pt,\lambda}(f_k\delta^k)=0$, for all $f_k\in\coeffalg$, and all $k\in\Z$ such that $p\nmid k$, and also that $\st_{\pt,\lambda}(f_k\delta^k)=f_k(\pt)\lambda^{\frac{k}{p}}$ if $p\mid k$. However, if $k\in\Z$, $f_k\in\coeffalg$, and $f_k\delta^k\in\fscomm$, so that $\supp(f_k)\subset\fixk$, then we must have $f_k(\pt)=0$ if $k\neq 0$, since $\pt\notin\intfixedpoints$. All in all, we conclude that, for $\sum_k f_k\delta^k\in\fscomm$, $\st_{\pt,\lambda}(\sum_k f_k\delta^k)=f_0(\pt)=\ch_{\pt}(\sum_k f_k\delta^k)$, where $\ch_{\pt}$ is as in the first part of Lemma~\ref{l:seven}. As in the previous two cases, we conclude that $\st_{\pt,\lambda_0}$ is multiplicative on $\cstarcomm$, hence that the restriction of $\st_{\pt,\lambda_0}$ to $\cstarcomm$ is, in fact, a pure state of $\cstarcomm$.

\medskip
As a consequence of the above three cases, we see that the pure states $\st_{\pt}$, for $\pt\in\aperpoints$, and $\st_{\pt,\lambda}$, for $\pt\in\perpoints$ and $\lambda\in\T$, restrict to pure states of $\cstarcomm$, as announced earlier. As a further consequence, we have the following result, which will be improved in Theorem~\ref{t:enveloping_algebra}.

\begin{prop}\label{p:homeomorphism_of_max_ideal_spaces}
The map, given by restricting a character of $\cstarcomm$ to $\lonecomm$, yields a homeomorphism between the maximal ideal space \ulb i.e., the pure state space\urb\ of the commutative Banach algebra $\cstarcomm$ and the maximal ideal space \ulb which coincides with the pure state space\urb\ of $\lonecomm$.
\end{prop}

\begin{proof}
The restriction map is obviously continuous, and it is also injective, since $\fscomm=\fscomm$ is dense in $\cstarcomm$. Hence it is sufficient to show surjectivity, and for this we combine the description of $\maxidealspacelonecomm$ in Proposition~\ref{p:nine} with the results in the three cases preceding the present proposition.

In the first case, if $\pt\in\aperpoints$, then the density of $\fscomm$ implies that, when restricting the restriction of $\st_{\pt}$ to $\cstarcomm$ further to $\lonecomm$, one obtains $\ch_{\pt}$ as in the first part of Proposition~\ref{p:nine}.

In the second case, if $\lambda\in\T$, and $\pt\in\perpoints$, with $\pt\in\intfixedpoints$ of period $p$, then, when restricting the restriction of $\st_{\pt,\lambda}$ to $\cstarcomm$ further to $\lonecomm$, one obtains $\ch_{\pt,\lambda^{\frac{n}{p}}}$, where, for $c\in\T$, $\ch_{\pt,c}$ is as in the second part of Proposition~\ref{p:nine}.

In the third case, if $\lambda$ in $\T$, and $\pt\in\perpoints$, with $\pt\notin\intfixedpoints$, then, when restricting the restriction of $\st_{\pt,\lambda}$ to $\cstarcomm$ further to $\lonecomm$, one obtains $\ch_{\pt}$ as in the first part of Proposition~\ref{p:nine}.

All in all, we see that the characters in the first part of Proposition~\ref{p:nine} are all obtained from the combined first and third cases, and that the characters in the second part of Proposition~\ref{p:nine} are all obtained from the second case.
\end{proof}

In view of Proposition~\ref{p:homeomorphism_of_max_ideal_spaces}, we will identify the pure states (characters) of $\cstarcomm$ and $\lonecomm$, writing $\ch_{\pt}$ (for $\pt\notin\intperpoints$) and $\ch_{\pt,c}$ (for $\pt\in\intperpoints$ and $c\in\T$) for both.

We summarise most of our findings in the following theorem, which also describes the fibers under restriction maps as we have, in fact, determined them during the discussion of the Cases~1 through~3 and the proof of Proposition~\ref{p:homeomorphism_of_max_ideal_spaces}.

\begin{thm}\label{t:summarising_theorem}\quad
\begin{enumerate}[\upshape (i)]
\item The pure state space of $\lonecomm$ and its maximal ideal space $\maxidealspacelonecomm$ coincide. Restricting pure states of $\lonecomm$ to $\coeffalg$ yields a continuous map from the pure state space  $\maxidealspacelonecomm$ onto the pure state space $\Delta(\coeffalg)$, and the fibers of this surjection are given in Proposition~\ref{p:nine}.

\item Restricting pure states of $\cstarcomm$ to $\lonecomm$ yields a homeomorphism between the pure state spaces $\maxidealspacecstarcomm$ and $\maxidealspacelonecomm$.

\item If $\pt\in \topspace$, then the restriction of each pure state extensions of $\ev_{\pt}$ to $\cstar$ to $\cstarcomm$ is a pure state of $\cstarcomm$.

\item The following diagrams describe, depending on the dynamical properties of $\pt\in\topspace$, how the pure state extensions of $\ev_{\pt}$ to $\cstar$ restrict down the chain in \eqref{e:inclusions}.
    \\ In these diagrams, the pure state at the top of the right hand side is the general pure state extension of $\ev_{\pt}$ to $\cstar$, and where the cardinality of the fibers has also been indicated.

\medskip
\noindent\emph{Case 1:} $\pt\in\aperpoints$.
\begin{equation*}
\xymatrix
{
\cstar&\quad&\st_{\pt}\ar[d]^{1-1}\\
\cstarcomm\ar @{^{(}->}[u]&\quad&\ch_{\pt}\ar[d]^{1-1}\\
\lonecomm\ar @{^{(}->}[u]^{\text{\textup{dense}}}&\quad&\ch_{\pt}\ar[d]^{1-1}\\
\coeffalg\ar @{^{(}->}[u]&\quad&\ev_{\pt}
}
\end{equation*}

\medskip
\noindent\emph{Case 2:} $\pt\in\perpoints$, \emph{with }$\pt\in\intfixedpoints$; $\lambda\in\T$\textup{ arbitrary}.
\\Let $n$ be the minimal $n\geq 1$ such that $\pt\in\fixnint$, and let $p\geq 1$ be the period of $\pt$.
\begin{equation*}
\xymatrix
{
\cstar&\quad&\st_{\pt,\lambda}\ar[d]^{\frac{n}{p}\textup{ to }1}\\
\cstarcomm\ar @{^{(}->}[u]&\quad&\ch_{\pt,\lambda^{\frac{n}{p}}}\ar[d]^{1-1}\\
\lonecomm\ar @{^{(}->}[u]^{\text{\textup{dense}}}&\quad&\ch_{\pt,\lambda^{\frac{n}{p}}}\ar[d]^{\T\textup{ to }1}\\
\coeffalg\ar @{^{(}->}[u]&\quad&\ev_{\pt}
}
\end{equation*}

\medskip
\noindent\emph{Case 3:} $\pt\in\perpoints$, \emph{with }$\pt\notin\intfixedpoints$; $\lambda\in\T$\textup{ arbitrary}.
\begin{equation*}
\xymatrix
{
\cstar&\quad&\st_{\pt,\lambda}\ar[d]^{\T\textup{ to }1}\\
\cstarcomm\ar @{^{(}->}[u]&\quad&\ch_{\pt}\ar[d]^{1-1}\\
\lonecomm\ar @{^{(}->}[u]^{\text{\textup{dense}}}&\quad&\ch_{\pt}\ar[d]^{1-1}\\
\coeffalg\ar @{^{(}->}[u]&\quad&\ev_{\pt}
}
\end{equation*}

\end{enumerate}
\end{thm}

The following result will be needed in the proof of the uniqueness statement for projections in Theorem~\ref{t:cstar_projection_unique}.

\begin{cor}\label{c:unique_extension}\quad
\begin{enumerate}[\upshape (i)]
\item For $\pt\in\aperpoints$, the pure state $\ch_{\pt}$ of $\cstarcomm$ has only $\st_{\pt}$ as state extension to $\cstar$.
\item For $\pt\in\intperpoints$, and $c\in\T$, the pure state $\ch_{\pt,c}$ of $\cstarcomm$ has only $\st_{\pt,c}$ as state extension to $\cstar$.
\end{enumerate}
\end{cor}

\begin{proof}
As to the first part, it is immediate from Case 1 in Theorem~\ref{t:summarising_theorem} that $\ch_{\pt}$ has only one pure state extension to $\cstar$, hence also only one state extension. For the second part, Case~2 of Theorem~\ref{t:summarising_theorem} applies. If $p$ is the period of $\pt$, then clearly the minimal $n\geq 1$ such that $\pt\in\fixnint$ is $p$ itself, so that the top arrow in the right hand side of the diagram for Case~2 is injective. Therefore, $\ch_{\pt,c}$ has only one pure state extension to $\cstar$, hence also only one state extension.
\end{proof}

\begin{rem}\label{r:separating_cstarcomm}
 \quad\begin{enumerate}[\upshape (i)]
 \item Corollary~\ref{c:separating_lonecomm} and Proposition~\ref{p:homeomorphism_of_max_ideal_spaces} imply that the pure states of $\cstarcomm$ in the first part of Corollary~\ref{c:unique_extension}, together with those in the second part, constitute a dense subset in $\maxidealspacecstarcomm$, hence separate the points of $\cstarcomm$. As Corollary~\ref{c:unique_extension} itself, this will be used in the proof of the uniqueness statement for projections in Theorem~\ref{t:cstar_projection_unique}.
 \item More generally, in view of the homeomorphism in Proposition~\ref{p:homeomorphism_of_max_ideal_spaces}, the paragraph preceding Corollary~\ref{c:separating_lonecomm} is now seen to provide a tool to generate dense (hence separating) subsets of $\maxidealspacecstarcomm$.
 \item In \cite[Lemma~3.5]{svensson_tomiyama} it was observed that the (unique) pure state extension to $\cstar$ of the pure states of $\maxidealspacecstarcomm$ in the first part of this remark are total on $\cstar$.  Hence this involves a larger algebra, but then only its positive cone.
 \end{enumerate}
\end{rem}

\section{The $C^*$-envelope of $\lonecomm$ and projections onto $\cstarcomm$ and $\lonecomm$}\label{s:enveloping_algebra_and_projections}

The results in Section~\ref{s:maxidealspace_lone} and~\ref{s:extending_and_restricting} enable us to clarify the relation between $\lone$ and its enveloping $C^*$-algebra $\cstar$ a bit further in Theorem~\ref{t:enveloping_algebra}, and also provide the tools to study the existence and uniqueness of projections onto $\cstarcomm$ and $\lonecomm$ in detail, cf.\ Theorems~\ref{t:cstar_projection_unique} and~\ref{t:lone_projection_unique}

To begin with, we have already seen in Proposition~\ref{p:homeomorphism_of_max_ideal_spaces} that the pure state spaces of $\cstarcomm$ and $\lonecomm$ are homeomorphic via a restriction map. In fact, more is true.

\begin{thm}\label{t:enveloping_algebra}
\!The $C^*\!$-algebra $\cstarcomm$ is the\! enveloping $C^*\!$-algebra of $\lonecomm$.
\end{thm}

\begin{proof}
We already know that $\lonecomm$ is a dense subalgebra of $\cstarcomm$, since $\fscomm\hookrightarrow\lonecomm\hookrightarrow\cstarcomm$. Hence all we need to show is that the enveloping $C^*$-norm of the unital involutive Banach algebra $\lonecomm$ coincides with the restriction of the norm of $\cstarcomm$ to $\lonecomm$. According to \cite[2.7.1]{dixmier}, the enveloping $C^*$-norm $\Vert\ell \Vert^\prime$ of an element $\ell$ of $\lonecomm$ can be calculated as
\[
\Vert \ell\Vert^\prime=\sup_{\st\in P} \st(\ell^*\ell)^{1/2},
\]
where $P$ is the pure state space of $\lonecomm$. We have already observed (cf.\ Theorem~\ref{t:summarising_theorem}) that this pure state space coincides with $\maxidealspacelonecomm$, hence
\[
\Vert \ell\Vert^\prime=\sup_{\ch\in\maxidealspacelonecomm}| \ch(\ell)|.
\]
Invoking Proposition~\ref{p:homeomorphism_of_max_ideal_spaces} then yields
\[
\Vert \ell\Vert^\prime=\sup_{\ch\in\maxidealspacecstarcomm}| \ch(c)|.
\]
By the commutative Gelfand-Naimark theorem, this is indeed equal to the norm of $\ell$ in the $C^*$-algebra $\cstarcomm$, as was to be proved.
\end{proof}

Thus, after the fact, the homeomorphism in Proposition~\ref{p:homeomorphism_of_max_ideal_spaces} is explained by Theorem~\ref{t:enveloping_algebra} and the general correspondence theorems for the states of an involutive Banach algebra with a 1-bounded approximate identity and the states of its enveloping $C^*$-algebra, cf.\ \cite[2.7.4 and~2.7.5]{dixmier}.

\medskip
We now turn to the existence and uniqueness of projections from $\cstar$ onto $\cstarcomm$, and from $\lone$ onto $\lonecomm$. This topic will occupy us for the remainder of this section. We start with existence in $\cstar$.

\begin{thm}\label{t:cstar_projection} The following are equivalent for $\cstar$:
\begin{enumerate}[\upshape (i)]
\item There exists a not necessarily continuous linear map $\Eprimestar:\fs\to\cstarcomm$, which is the identity on $\fscomm$, and which is a morphism of left $\coeffalg$-modules.
\item There exists a not necessarily continuous linear map $\Eprimestar:\fs\to\cstarcomm$, which is the identity on $\fscomm$, and which is a morphism of right $\coeffalg$-modules.
\item The sets $\fixkint$ are closed, for all $k\in\Z$.
\item There exists a faithful positive norm one projection $\Eprimestar:\cstar\to\cstarcomm$ onto $\cstarcomm$, which is a morphism of $\cstarcomm$-bimodules, and which restricts to a projection from $\fs$ onto $\fscomm$.
\end{enumerate}
\end{thm}

\begin{proof}
We start by showing that (1) implies (3). Fix $k\in\Z$; we may assume that $\fixk\neq\emptyset$. Suppose that $\pt\in\overline{\fixkint}$, and let $(\ptx_\alpha)$ be a net in $\fixkint$ converging to $\pt$. For each $\alpha$, choose $h_\alpha\in\coeffalg$, such that $h_\alpha(\ptx_\alpha)=1$, and $\supp (h_\alpha)\subset\fixk$. Then $h_\alpha\delta^{k}\in\cstarcomm$, for all $\alpha$. Since $\Eprimestar$ is a morphism of left $\coeffalg$-modules, we have $h_\alpha\delta^{k}=\Eprimestar(h_\alpha\delta^{k})=h_\alpha \Eprimestar(\delta^{k})$, for all $\alpha$. Using this and the first part of \eqref{e:projection_properties}, we find that $h_\alpha=(h_\alpha\delta^{k})(k)=[h_\alpha\Eprimestar(\delta^{k})](k)=h_\alpha[\Eprimestar(\delta^{k})(k)]$, for all $\alpha$. Evaluation at $\pt$ yields $[\Eprimestar(\delta^{k})(k)](\ptx_\alpha)=1$, for all $\alpha$, hence $[\Eprimestar(\delta^{k})(k)](\pt)=1$ by continuity. Therefore, $\pt$ is an interior point of the support of $[\Eprimestar(\delta^{k})(k)]$. Since this support is contained in $\fix_{k}(\homeo)$, we conclude that $\pt\in\fixkint$, as was to be shown.

The proof that (2) implies (3) is similar, and we use the same notation. Since $\delta^{k}h_\alpha=h_\alpha\delta^{k}\in\cstarcomm$, and $\Eprimestar$ is a morphism of right $\coeffalg$-modules, we have $\delta^{k}h_\alpha=\Eprimestar(\delta^{k}h_\alpha)=\Eprimestar(\delta^{k})h_\alpha$, for all $\alpha$.
Using the second part of \eqref{e:projection_properties}, this implies that $h_\alpha=(h_\alpha\delta^{k})(k)=(\delta^{k}h_\alpha)(k)=[\Eprimestar(\delta^{k})h_\alpha](k)=(h_\alpha\circ\homeo^{-k})[\Eprimestar(\delta^{k})(k)]=h_\alpha[\Eprimestar(\delta^{k})(k)]$, for all $\alpha$. Evaluation at $\pt$ again yields $[\Eprimestar(\delta^{k})(k)](\ptx_\alpha)=1$, for all $\alpha$, and as before this implies that $\pt\in\fixkint$.

In order to show that (3) implies (4), we let $\chi_k$ be the characteristic function of $\fixkint$, for $k\in\Z$. Since $\fixkint$ is clopen, $\chi_k\in \coeffalg$, for all $k\in\Z$, hence we can define $\Eprimestar:\fs\to\cstarcomm$ by $\Eprimestar(\sum_k f_k\delta^k):=\sum_k \chi_k f_k\delta^k$, for $\sum_k f_k\delta^k\in\fs$. It is easy to see that the image of $\Eprimestar$ is, in fact, contained in $\cstarcomm$, hence in $\fscomm$, and that $\Eprimestar$ is the identity on $\fscomm$. We start by showing that, when $\fs$ carries the norm from $\cstar$, $\Eprimestar$ is continuous, has norm at most 1, is positive, and is a morphism of $\fscomm$-bimodules. Then the continuous extension of $\Eprimestar$ to $\cstar$ will provide a projection from $\cstar$ onto $\lonecomm$ for which all required properties are clear with the exception of faithfulness, which we will then consider later.
\\As to the continuity and the contractivity of $\Eprimestar$, since $\cstarcomm$ is a commutative $C^*$-algebra, we need to show that $|\psi(\sum_k \chi_k f_k\delta^k|\leq \Vert \sum_k f_k\delta^k\Vert$, for all $\psi\in\maxidealspacecstarcomm$, and all $\sum_k f_k\delta^k\in\fs$. Since Proposition~\ref{p:homeomorphism_of_max_ideal_spaces} and Theorem~\ref{t:topological_quotient} provide a surjective map from $\topspace\times\T$ onto $\maxidealspacecstarcomm$, we see that we are left with demonstrating that
$|\sum_k \chi_k(\ptx)f_k(\ptx) z^k|\leq \Vert \sum_k f_k\delta^k\Vert$, for all $\sum_k f_k\delta^k\in\fs$, $\ptx\in \topspace$, and $z\in \T$. In that case,  $|\sum_k \chi_k(\ptx)f_k(\ptx) z^k|\leq\sum_k \Vert\chi_k f_k\Vert_\infty\leq \sum_k \Vert f_k\Vert_\infty=\Vert \sum_k f_k\delta^k\Vert_1$, and since $\cstar$ is the enveloping $C^*$-algebra of $\lone$, we certainly have $\Vert \sum_k f_k\delta^k\Vert_1\leq \Vert \sum_k f_k\delta^k\Vert$. Hence $\Eprimestar$ is contractive.
\\As to the positivity, if $\sum_k f_k\delta^k\in\fs$, then a short calculation yields
\begin{equation}\label{e:sumstarsumfinitesequences}
\Eprimestar\left[\left(\sum_k f_k\delta^k\right)^*\left(\sum_k f_k\delta^k\right)\right]=\sum_m \left[\chi_m\sum_k \left(\overline{f}_k f_{k+m}\right)\circ\homeo^k\right]\delta^m.
\end{equation}
In order to see that the right hand side is positive in the commutative $C^*$-algebra $\cstarcomm$, we apply its characters as described in Proposition~\ref{p:nine}. To begin with, if $\pt\notin\intfixedpoints$, then
\begin{equation}\label{e:cstar_positive_point_evaluation}
\ch_{\pt}\left(\sum_m \left[\chi_m\sum_k \left(\overline{f}_k f_{k+m}\right)\circ\homeo^k\right]\delta^m\right)=\sum_k |f_k(\homeo^k \pt)|^2\geq 0.
\end{equation}
If $\pt\in\intfixedpoints$, let $n$ be the minimal $n\geq 1$ such that $\pt\in\fixnint$. Then, for $c\in\T$,
\begin{align}\label{e:cstar_positive_point_evaluation_and_torus}
\ch_{\pt,c}\left(\!\sum_m\! \left[\chi_m\sum_k \left(\overline{f}_k f_{k+m}\right)\circ\homeo^k\right]\!\delta^m\!\right)&=\sum_j \left[\chi_{jn}(\pt)\sum_k \left(\overline{f}_k f_{k+jn}\right)(\homeo^k \pt)\right]c^{j}\\
&=\sum_j \sum_k \left(\overline{f}_k f_{k+jn}\right)(\homeo^k \pt)c^{j}\notag\\
&=\sum_{r=0}^{n-1}\sum_{i,j}\left(\overline{f}_{r+in} f_{r+(i+j)n}\right)(\homeo^{r+in}\pt)c^j\notag\\
&=\sum_{r=0}^{n-1}\sum_{i,j}\left(\overline{f}_{r+in} f_{r+jn}\right)(\homeo^{r}\pt)c^{j-i}\notag\\
&=\sum_{r=0}^{n-1}\left|\sum_{j} f_{r+jn}(\homeo^r \pt) c^j \right|^2
\notag\\&\geq 0.\notag
\end{align}
Hence $\Eprimestar$ is positive, and it remains to show that $\Eprimestar:\fs\to\cstarcomm$ is a morphism of $\fscomm$-bimodules. Let $k,l\in\Z$, $f_k\in\coeffalg$, and $g_l\in\coeffalg$, with $\supp(g_l)\subset\fixl$. We need to prove that $\Eprimestar(f_k\delta^k\cdot g_l\delta^l)=\left[\Eprimestar(f_k\delta^k)\right]g_l\delta^l$ and $\Eprimestar(g_l\delta^l\cdot f_k\delta^k)=g_l\delta^l\Eprimestar(f_k\delta^k)$. Starting with the first relation, this is easily seen to be equivalent with
\begin{equation}\label{e:right_module_morphism}
\chi_{k+l}(\ptx)f_k(\ptx) g_l(\homeo^{-k}\ptx)=\chi_{k}(\ptx)f_k(\ptx) g_l(\homeo^{-k}\ptx)
\end{equation}
being valid for all $\ptx\in \topspace$. If $g_l(\homeo^{-k}\ptx)=0$, then \eqref{e:right_module_morphism} clearly holds. If $g_l(\homeo^{-k}\ptx)\neq 0$, then $\homeo^{-k}\ptx\in\supp(g_l)^\intsymbol\subset\fixlint$, hence $\ptx\in\fixlint$. It is easily seen that this implies $\chi_{k+l}(\ptx)=\chi_k(\ptx)$, so that \eqref{e:right_module_morphism} holds once more. Hence $\Eprimestar$ preserves the right $\fscomm$-action. The requirement $\Eprimestar(g_l\delta^l\cdot f_k\delta^k)=g_l\delta^l\Eprimestar(f_k\delta^k)$ is equivalent with
\begin{equation}\label{e:left_module_morphism}
\chi_{k+l}(\ptx)f_k(\homeo^{-l}\ptx)g_l(x)=\chi_k(\homeo^{-l}\ptx)f_k(\homeo^{-l}x)g_l(\ptx)
\end{equation}
holding for all $\ptx\in \topspace$. If $g_l(\ptx)=0$, then this is clear. If $g_l(\ptx)\neq 0$, then $\ptx\in\supp(g_l)^\intsymbol\subset\fixlint$, and \eqref{e:left_module_morphism} reduces to $\chi_{k+l}(\ptx)f_k(\ptx)g_l(x)=\chi_k(\ptx)f_k(x)g_l(\ptx)$. Since $x\in\fixlint$ implies that $\chi_{k+l}(x)=\chi_k(x)$ again, this indeed holds true. Hence $\Eprimestar$ preserves the left $\fscomm$-action.

As already mentioned, the continuous extension of $\Eprimestar$ to $\cstar$ satisfies all requirements, except that we still have to observe that it is faithful. Now $\Estar\circ\Eprimestar=\Estar$ on $\fs$ by construction, and by continuity this holds on $\cstar$ as well. Therefore the known faithfulness of $\Estar$ on $\cstar$ implies that of $\Eprimestar$.

It is clear that (4) implies both (1) and (2).
\end{proof}

\begin{rem}\quad
\begin{enumerate}[\upshape (i)]
\item It is interesting that preservation of a one-sided $\coeffalg$-action in (1) and (2) implies preservation of a two-sided action of a larger algebra $\cstarcomm$ in (4).
\item The proof that (3) implies (4) in Theorem~\ref{t:cstar_projection} could be simplified somewhat by merely establishing that $\Eprimestar:\cstar\to\cstarcomm$ as constructed is a norm one projection, and then invoking the projection theorem for $C^*$-algebras (cf.\ \cite{tomiyama_projection_theorem}, \cite[II.6.10.2]{blackadar}) to conclude that $\Eprimestar$ is automatically positive and a morphism of $\cstarcomm$-bimodules. Likewise, (4) could be replaced by an equivalent statement, which merely requires the existence of a faithful norm one projection from $\cstar$ onto $\cstarcomm$, projecting $\fs$ onto $\fscomm$. We have preferred a direct proof, and we have also included the ``redundant''  properties in (4), in order to bring out the parallel with the corresponding statements in Theorem~\ref{t:lone_projection} for $\lone$, where such a general projection theorem is not available.
\end{enumerate}
\end{rem}

We use Theorem~\ref{t:cstar_projection} and its proof in our next result on uniqueness for $\cstar$.\footnote{The equivalence of (1) and (2) and the uniqueness and faithfulness of a norm one projection as in Theorem~\ref{t:cstar_projection_unique} were announced in \cite[Theorem~6.1]{svensson_tomiyama}, but with a misprint. In the notation of \cite{svensson_tomiyama}, which is slightly different from ours, $\perk^0$ in Theorem~6.1 should be replaced with $\per^k(\homeo)^0$.}

\begin{thm}\label{t:cstar_projection_unique} The following are equivalent for $\cstar$:
\begin{enumerate}[\upshape (i)]
\item The sets $\fixkint$ are closed, for all $k\in\Z$.
\item There exists a norm one projection from $\cstar$ onto $\cstarcomm$.
\item There exists a positive projection from $\cstar$ onto $\cstarcomm$.
\end{enumerate}
In that case, a projection as in \ulb 2\urb\ or \ulb 3\urb\ is unique. Denoting it by $\Eprimestar$, and denoting the continuous characteristic function of $\fixkint$ by $\chi_k$, for $k\in\Z$, it is given on the dense involutive subalgebra $\fs$ of $\cstar$ by
\begin{equation}
\Eprimestar\left(\sum_k f_k\delta^k\right)=\sum_k \chi_k f_k\delta^k,
\end{equation}
for $\sum_k f_k\delta^k\in\fs$. Furthermore, $\Eprimestar$ is a faithful positive norm one projection, which is a morphism of $\cstarcomm$-bimodules, and which restricts to a projection from $\fs$ onto $\fscomm$.
\end{thm}

\begin{proof}
It is clear from Theorem~\ref{t:cstar_projection} that (1) implies (2), and that (3) implies (1). By the projection theorem for $C^*$-algebras (cf.\ \cite{tomiyama_projection_theorem}, \cite[II.6.10.2]{blackadar}), any projection as in (2) is a projection as in (3). Hence (2) implies (3) and, moreover, for the uniqueness statements it is sufficient to establish the uniqueness for projections as in (3). As to this, assume that $\Eprimestar$ and $\Fprimestar$ are two positive projections from $\cstar$ onto $\cstarcomm$. If $\pt\in\aperpoints$, then $\ch_{\pt}\circ\Eprimestar$ and $\ch_{\pt}\circ\Fprimestar$ are two state extensions of $\ch_{\pt}$ from $\cstarcomm$ to $\cstar$. According to Corollary~\ref{c:unique_extension}, there is only one such, hence $\ch_{\pt}\circ\Eprimestar=\ch_{\pt}\circ\Fprimestar$, for all $\pt\in\aperpoints$. Likewise, if $\pt\in\intperpoints$, and $c\in\T$, then, again by Corollary~\ref{c:unique_extension}, it follows that $\ch_{\pt,c}\circ\Eprimestar=\ch_{\pt,c}\circ\Fprimestar$. Since we know from Remark~\ref{r:separating_cstarcomm} that the states $\ch_{\pt}$, for $\ptx\in\aperpoints$, together with all $\ch_{\pt,c}$, for $\ptx\in\intperpoints$ and $c\in\T$, separate the points in $\lonecomm$, we conclude that $\Eprimestar=\Fprimestar$, as desired. Now that the uniqueness of positive projections has been established, the rest is immediate, since then the projection $\Eprimestar$ in Theorem~\ref{t:cstar_projection} must be this positive projection. The remainder of the present theorem therefore follows from Theorem~\ref{t:cstar_projection} and its proof.
\end{proof}

\begin{rem}
For a topologically free system, $\cstarcomm=\coeffalg$ (and vice versa). For such a system, Theorem~\ref{t:cstar_projection_unique} implies that $\Estar$ is the unique norm one projection from $\cstar$ onto $\coeffalg$, so that we have retrieved this previously known result \cite[Proposition~2.11]{tomiyama_notes_two} as a special case.
\end{rem}

We now turn to $\lone$. Part (4) of the following companion result of Theorem~\ref{t:cstar_projection} uses a notion of positivity in the algebra $\lone$. As usual, we define the positive cone $\loneposcone$ in $\lone$ as the set of all elements of the form $\sum_j l_j^*l_j$,  where $l_j\in\lone$, for all $j$, and the summation is finite. This is a convex cone and it is also proper, as becomes obvious when viewing $\lone$ as an involutive subalgebra of the $C^*$-algebra $\cstar$. Hence this cone induces a partial ordering on the self-adjoint elements of $\lonecomm$. As usual, a linear map from $\lone$ to a vector space will then be called faithful if 0 is the only positive element in its kernel.

\begin{thm}\label{t:lone_projection} The following are equivalent for $\lone$:
\begin{enumerate}[\upshape (i)]
\item There exists a not necessarily continuous linear map $\Eprimeone:\fs\to\lonecomm$, which is the identity on $\fscomm$, and which is a morphism of left $\coeffalg$-modules.
\item There exists a not necessarily continuous linear map $\Eprimeone:\fs\to\lonecomm$, which is the identity on $\fscomm$, and which is a morphism of right $\coeffalg$-modules.
\item The sets $\fixkint$ are closed, for all $k\in\Z$.
\item There exists a faithful involutive norm one projection $\Eprimeone:\lone\to\lonecomm$ onto $\lonecomm$, which is a morphism of $\lonecomm$-bimodules, which restricts to a projection from $\fs$ onto $\fscomm$, and which is positive in the sense that $\phi(\Eprimeone(\ell))\geq 0$, for all $\ell\in\loneposcone$, and all states $\phi$ of $\lonecomm$.
\end{enumerate}
\end{thm}

\begin{proof}
The proof that each of (1) and (2) implies (3) is a simplified version (since no projection is needed to define coefficients) of the proof of the analogous implications in Theorem~\ref{t:cstar_projection}, and is therefore omitted.

In order to show that (3) implies (4), we let $\chi_k$ be the characteristic function of the clopen set $\fixkint$ again, for $k\in\Z$. Instead of arguing by density, as in the proof of Theorem~\ref{t:cstar_projection}, we can now directly define $\Eprimeone:\lone\to\lonecomm$ by $\Eprimeone(\sum_k f_k\delta^k):=\sum_k \chi_k f_k\delta^k$, for $\sum_k f_k\delta^k\in\lone$. It is easily checked that, in fact, $\Eprimeone$ maps $\lone$ into $\lonecomm$, and that it is the identity on $\lonecomm$. Clearly it restricts to a projection from $\fs$ onto $\fscomm$. It is obvious from the definition that $\Eprimestar$ is a norm one projection, and it follows as in the proof of Theorem~\ref{t:cstar_projection} that it preserves the left and right $\lonecomm$-action. It is routine to verify that $\Eprimeone$ is involutive. As to the faithfulness, we now have an analogue of \eqref{e:sumstarsumfinitesequences} with infinite summations, namely
\begin{equation}\label{e:sumstarsumlone}
\Eprimeone\left[\left(\sum_k f_k\delta^k\right)^*\left(\sum_k f_k\delta^k\right)\right]=\sum_m \left[\chi_m\sum_k \left(\overline{f}_k f_{k+m}\right)\circ\homeo^k\right]\delta^m,
\end{equation}
for $\sum_k f_k\delta^k\in\lone$. The coefficient of $\delta^0$ in the right hand side of \eqref{e:sumstarsumlone} is $\sum_k |f_k\circ\homeo^k|^2$. Therefore, if the left hand side in \eqref{e:sumstarsumlone} is zero, then $\sum_k f_k\delta^k=0$, and this argument obviously extends to the finite sums in the definition of the positive cone of $\lone$. To check the positivity in the sense as stated we may assume that $\phi$ is pure. From the first part of Theorem~\ref{t:summarising_theorem} and Proposition~\ref{p:nine} we know what these pure states of $\lonecomm$ are, and then the positivity follows from the obvious calculations parallelling \eqref{e:cstar_positive_point_evaluation} and \eqref{e:cstar_positive_point_evaluation_and_torus}, for the infinite summations in \eqref{e:sumstarsumlone}.

It is clear that (4) implies both (1) and (2).
\end{proof}

Uniqueness is considered in the next result.

\begin{thm}\label{t:lone_projection_unique} The following are equivalent for $\lone$:
\begin{enumerate}[\upshape (i)]
\item The sets $\fixkint$ are closed, for all $k\in\Z$.
\item There exists a not necessarily continuous projection from $\lone$ onto $\lonecomm$, which is positive in the sense that $\phi(\Eprimeone(\ell))\geq 0$, for all $\ell\in\loneposcone$, and all states $\phi$ of $\lonecomm$.
\end{enumerate}
In that case, a projection as in \ulb 2\urb\ is unique. Denoting it by $\Eprimeone$, and denoting the continuous characteristic function of $\fixkint$ by $\chi_k$, for $k\in\Z$, it is given by
\begin{equation}
\Eprimeone\left(\sum_k f_k\delta^k\right)=\sum_k \chi_k f_k\delta^k,
\end{equation}
for $\sum_k f_k\delta^k\in\lone$. Furthermore, $\Eprimeone$ is a faithful involutive norm one projection, which is a morphism of $\lonecomm$-bimodules, and which restricts to a projection from $\fs$ onto $\fscomm$.
\end{thm}

\begin{proof}
It is clear from Theorem~\ref{t:lone_projection} that (1) implies (2), and that (2) implies (1). Just as in the proof of Theorem~\ref{t:cstar_projection_unique}, we need only show that a positive projection as in (2) is unique, because it must then be the positive projection in part (4) of Theorem~\ref{t:lone_projection}.
For this, we use a slight modification of the earlier argument. Assume that $\Eprimeone$ and $\Fprimeone$ are two projections from $\lone$ onto $\lonecomm$, positive as indicated. If $\pt\in\aperpoints$, then $\ch_{\pt}\circ\Eprimeone$ and $\ch_{\pt}\circ\Fprimeone$ are two state extensions of $\ch_{\pt}$ from $\lonecomm$ to $\lone$. These extensions, in turn, can be extended to states $\phi_{E,\pt}$ and $\phi_{F,\pt}$ on the enveloping $C^*$-algebra $\cstar$ of $\lone$. Both states $\phi_{E,\pt}$ and $\phi_{F,\pt}$ of $\cstar$ then extend the state $\ch_{\pt}$ of $\cstarcomm$. However, according to the first part of Corollary~\ref{c:unique_extension}, there is only one such state extension, hence $\phi_{E,\pt}=\phi_{F,\pt}$. We conclude from this that $\ch_{\pt}\circ\Eprimeone=\ch_{\pt}\circ\Fprimeone$, for all $\pt\in\topspace$. Likewise, it $\pt\in\intperpoints$, and $c\in\T$, then the second part of Corollary~\ref{c:unique_extension} implies that $\ch_{\pt,c}\circ\Eprimeone=\ch_{\pt,c}\circ\Fprimeone$. Corollary~\ref{c:separating_lonecomm} then shows that $\Eprimeone=\Fprimeone$.
\end{proof}

\begin{rem}
The notions of positivity for $\lone$ and $\lonecomm$ as they occur in
Theorem~\ref{t:lone_projection} and~\ref{t:lone_projection_unique}
are not the same. For $\lone$, positivity is defined in terms of a convex cone---a variation, passing to it closure, would also be natural to consider---whereas, for $\lonecomm$, positivity is defined in terms of states. For $C^*$-algebras, all these notions coincide, but for general involutive Banach algebras they need not. The relations between these notions is a subtle one, as is, e.g., attested by the material in \cite{balachandran, doran_belfi, fragoulopoulou, palmer_two}. More research is needed to determine whether more symmetric versions of Theorem~\ref{t:lone_projection} and~\ref{t:lone_projection_unique}, involving cone-to-cone positivity, or state-to-state positivity, can be established, perhaps under additional conditions on the dynamics.
\end{rem}

\appendix

\section{Topological results on the periodic points}\label{s:topological_results}

The periodic points of the homeomorphism $\homeo:\topspace\to \topspace$ play an important role in the investigation of $\cstar$ and $\lone$, and consequently various topological results concerning these points have been established, scattered over a number of papers. In this appendix, we collect the optimal version of these results as we know them, establishing non-trivial criteria for topological freeness in Corollary~\ref{c:topological_freeness} and exhibiting non-trivial dense subsets in Corollary~\ref{c:density_in_topspace}. These two results are to some extent an improvement over what is already known, and it also seemed worthwhile to collect all material in this direction in one place. Moreover, it is instructive to see how they are easily inferred from a new and rather general statement on equal closures, Proposition~\ref{p:internal_structure}. The results are actually valid when $\topspace$ is a locally compact Hausdorff space, and it is in this context that we will again use the obvious notations $\fixq$, $\perq$, $\perpoints$, and $\aperpoints$, which were previously only defined for a compact Hausdorff space $\topspace$.

As a preparation, we recall \cite[Theorem~2.2]{rudin} that the category theorem is valid for a locally compact Hausdorff space $\topspace$: the countable intersection of open dense subsets of $\topspace$ is still dense. Consequently, if $\topspace\neq\emptyset$ is the countable union of subsets, then the closure of at least one of these subsets must have a non-empty interior. Furthermore, we recall that the locally compact (in the induced topology) subspaces of a locally compact Hausdorff space $\topspace$ are precisely the sets which are the intersection of an open subset of $\topspace$ and a closed subset of $\topspace$, cf.\ \cite[I.3.3 and I.9.7]{bourbaki}.

\begin{prop}\label{p:internal_structure}
Let $\topspace$ be a locally compact Hausdorff space, and $\homeo:\topspace\to\topspace$ a homeomorphism. Suppose $S\subset\{1,2,3,\ldots\}$ is a non-empty finite or infinite set, and let $P_S=\{p\in\{1,2,3,\ldots\} : p \textup{ divides some }s\in S\}$. Then
\begin{equation}\label{e:internal_structure_one}
\bigcup_{p\in P_S}\perpnewint \subset \bigcup_{s\in S}\fixsint \subset \left(\bigcup_{s\in S}\fixs\right)^\intsymbol\subset \overline{\bigcup_{p\in P_S}\perpnewint},
\end{equation}
\begin{equation}\label{e:internal_structure_one_and_a_halve}
\bigcup_{p\in P_S}\perpnewint\subset\left(\bigcup_{p\in P_S}\perpnew\right)^\intsymbol\subset\left(\bigcup_{s\in S}\fixs\right)^\intsymbol,
\end{equation}
and
\begin{equation}\label{e:internal_structure_two}
\overline{\left(\bigcup_{p\in P_S}\perpnew\right)^\intsymbol}=
\overline{\bigcup_{p\in P_S}\perpnewint}=
\overline{\bigcup_{s\in S}\fixsint}=
\overline{\left(\bigcup_{s\in S}\fixs\right)^\intsymbol}.
\end{equation}
\end{prop}

\begin{proof}
Clearly, \eqref{e:internal_structure_two} follows immediately from \eqref{e:internal_structure_one} and \eqref{e:internal_structure_one_and_a_halve}. As to the first inclusion in \eqref{e:internal_structure_one}: if $\ptx\in V\subset\per_{p_0}^\intsymbol(\homeo)$ for some $p_0\in P_S$ with $p_0\mid s_0\in S$, and $V$ open in $\topspace$, then $\ptx\in V\subset \per_{p_0}(\homeo)\subset\fix_{s_0}(\homeo)$, so that $\ptx\in\fix_{s_0}^\intsymbol(\homeo)$. This establishes the first inclusion in \eqref{e:internal_structure_one}, and a similar argument yields the second inclusion in \eqref{e:internal_structure_one_and_a_halve}. Since the second inclusion in \eqref{e:internal_structure_one} and the first inclusion in \eqref{e:internal_structure_one_and_a_halve} are obvious, we are left with the proof of the third inclusion in \eqref{e:internal_structure_one}. While doing so, we will use the notations $\interior_A(B)$ and ${\overline{B}}^A$ to denote the interior, respectively the closure, of a set $B\subset A\subset \topspace$ in the topological space $A$ with its induced topology from $\topspace$. The interior of $B\subset \topspace$ with respect to $\topspace$ will continue to be denoted by $B^\intsymbol$, the closure in $\topspace$ by $\overline B$ and the complement in $\topspace$ by $B^c$.

We may assume that $\left(\bigcup_{s\in S}\fixs\right)^\intsymbol\neq\emptyset$. Let $\ptx\in\left(\bigcup_{s\in S}\fixs\right)^\intsymbol$ and an open neighbourhood $V$ of $x$ in $\topspace$ be given. We must show that $V\cap \bigcup_{p\in P_S}\perpnewint\neq\emptyset$. Since $\left(\bigcup_{s\in S}\fixs\right)^\intsymbol$ is open in $\topspace$, we can replace $V$ with the open neighbourhood $V\cap \left(\bigcup_{s\in S}\fixs\right)^\intsymbol$ of $x$ in $\topspace$, and hence assume that $V$ is open in $\topspace$ and that $\ptx\in V\subset \left(\bigcup_{s\in S}\fixs\right)^\intsymbol$. It will be sufficient to prove that $V\cap \bigcup_{p\in P_S}\perpnewint\neq\emptyset$ for such $V$. In that case, certainly $V\subset \bigcup_{s\in S}\fixs$, hence $V=\bigcup_{s\in S}\left(V\cap \fixs\right)$, and since $V$ is a locally compact Hausdorff space in the induced topology, there exists $s_0\in S$, such that $\interior_V\left(\overline{V\cap\fix_{s_0}(\homeo)}^V\right)\neq\emptyset$. Since $\fix_{s_0}(\homeo)$ is closed in $\topspace$, we conclude that $\interior_V\left(V\cap\fix_{s_0}(\homeo)\right)\neq\emptyset$. Since certainly $s_0\in S\subset P_S$, we thus see that it is meaningful to define $p_0$ as the smallest element of $P_S$ with the property that $\interior_V\left(V\cap\fix_{p_0}(\homeo)\right)\neq\emptyset$. Let then $V^\prime$ be an open subset of $\topspace$, such that $\emptyset\neq V\cap V^\prime\subset V\cap\fix_{p_0}(\homeo)$. Then $V\cap V^\prime=V\cap V^\prime\cap\fix_{p_0}(\homeo)=V\cap V^\prime\cap\bigcup_{d\mid p_0}\per_d(\homeo)=\bigcup_{d\mid p_0}(V\cap V^\prime\cap\per_d(\homeo))$, and by the category theorem there exists a divisor $d_0$ of $p_0$ (note that then $d_0\in P_S$, since $p_0\in P_S$) such that $\interior_{V\cap V^\prime}\left(\overline{V\cap V^\prime\cap\per_{d_0}(\homeo)}^{V\cap V^\prime}\right)\neq\emptyset$. Since $\overline{V\cap V^\prime\cap\per_{d_0}(\homeo)}^{V\cap V^\prime}\subset \overline{V\cap V^\prime\cap\fix_{d_0}(\homeo)}^{V\cap V^\prime}=V\cap V^\prime\cap\fix_{d_0}(\homeo)$, we conclude that there exists an open subset $V^{\prime\prime}$ of $\topspace$, such that $\emptyset\neq V\cap V^\prime\cap V^{\prime\prime}\subset V\cap V^\prime\cap\fix_{d_0}(\homeo)$. In particular, $V\cap V^\prime\cap V^{\prime\prime}$ is a non-empty subset of $V$, open in $V$, and contained in $V\cap\fix_{d_0}(\homeo)$. Hence $\interior_V\left(V\cap\fix_{d_0}(\homeo)\right)\neq\emptyset$, and by the minimality property of $p_0$ we conclude that $d_0=p_0$. Hence $\interior_{V\cap V^\prime}\left(\overline{V\cap V^\prime\cap\per_{p_0}(\homeo)}^{V\cap V^\prime}\right)\neq\emptyset$. In particular, $V\cap V^\prime\cap \per_{p_0}(\homeo)\neq\emptyset$. As $V\cap V^\prime\subset\fix_{p_0}(\homeo)$, and $\per_{p_0}(\homeo)=\fix_{p_0}(\homeo)\cap\bigcap_{d\mid p_0,\,d\neq p_0}\fix_d(\homeo)^c$, we have
\begin{align*}
V\cap V^\prime\cap\per_{p_0}(\homeo)&=V\cap V^\prime\cap\fix_{p_0}(\homeo)\cap\bigcap_{d\mid p_0,\,d\neq p_0}\fix_d(\homeo)^c\\
&=V\cap V^\prime\cap \bigcap_{d\mid p_0,\,d\neq p_0}\fix_d(\homeo)^c.
\end{align*}
Consequently, $V\cap V^\prime\cap\per_{p_0}(\homeo)$ is open in $\topspace$, hence $V\cap V^\prime\cap \per_{p_0}(\homeo)\subset\per_{p_0}^\intsymbol(\homeo)$. Since $V\cap (V\cap V^\prime\cap\per_{p_0}(\homeo))\neq\emptyset$, we conclude that $V\cap\per_{p_0}^\intsymbol(\homeo)\neq\emptyset$. Consequently, $V\cap \bigcup_{p\in P_S}\perpnewint\neq\emptyset$, as was to be shown.

\end{proof}

In particular, if we take $S=\{1,2,3,\ldots\}$, we see that
\begin{equation}\label{e:internal_structure_three}
\overline{\left(\bigcup_{q\geq 1} \perq\right)^\intsymbol}=
\overline{\intperpoints}=
\overline{\intfixedpoints}=
\overline{\left(\bigcup_{q\geq 1}\fixq\right)^\intsymbol}.
\end{equation}

We can now combine \eqref{e:internal_structure_three} with other known results. For example, the equivalence of the first and second part in the following result is taken from \cite[Lemma~2.1]{de_jeu_svensson_tomiyama}, and we include the brief proof for completeness of the presentation.

\begin{cor}\label{c:topological_freeness}
Let $\topspace$ be a locally compact Hausdorff space, and $\homeo:\topspace\to\topspace$ a homeomorphism. Then the following are equivalent:
\begin{enumerate}[\upshape (i)]
\item $(\topspace,\homeo)$ is topologically free, i.e., $\aperpoints$ is dense in $\topspace$.
\item $\fixq$ has empty interior, for all $q\geq 1$.
\item $\bigcup_{q\geq 1}\fixq$ has empty interior.
\item $\perq$ has empty interior, for all $q\geq 1$.
\item $\bigcup_{q\geq 1}\perq$ has empty interior.
\end{enumerate}
\end{cor}

\begin{proof}
For the equivalence of the first and second part, note that $\aperpoints=\bigcap_{q=1}^\infty \fixq^c$. Since the sets $\fixq^c$ are all open, the category theorem yields that $\aperpoints$ is dense if and only if $\fixq^c$ is dense, for all $q\geq 1$, i.e., if and only if $\fixqint=\emptyset$, for all $q\geq 1$. The equivalence of the second, third, fourth and fifth part is immediate from \eqref{e:internal_structure_three}.
\end{proof}

The first part of the next result is also from \cite[Lemma~2.1]{de_jeu_svensson_tomiyama}, and we include the short proof again for the sake of completeness. The third part was established independently in \cite[Lemma~2.1]{svensson_tomiyama}, but will now be seen to be an immediate consequence of the first part and \eqref{e:internal_structure_three}.
\begin{cor}\label{c:density_in_topspace}
Let $\topspace$ be a locally compact Hausdorff space, and $\homeo:\topspace\to\topspace$ a homeomorphism. Then:
\begin{enumerate}[\upshape (i)]
\item The union of $\aperpoints$ and $\intfixedpoints$ is dense in $\topspace$.
\item The union of $\aperpoints$ and $\left(\bigcup_{q\geq 1} \fixq\right)^\intsymbol$ is dense in $\topspace$.
\item The union of $\aperpoints$ and $\intperpoints$ is dense in $\topspace$.
\item The union of $\aperpoints$ and $\left(\bigcup_{q\geq 1} \perq\right)^\intsymbol$ is dense in $\topspace$.
\end{enumerate}
\end{cor}

\begin{proof}
We start with the first part and use the same notational conventions as in the proof of Proposition~\ref{p:internal_structure}.
\[
Y=\overline{\aperpoints\cup\intfixedpoints}^{\,\,c}.
\]
Suppose that $Y\neq\emptyset$. Since $Y\subset\perpoints$, $Y=\bigcup_{q\geq 1} Y\cap\fixq$. The category theorem shows that there exists $q_0\geq 1$, such that $\interior_Y(\overline{Y\cap\fix_{q_0}(\homeo)}^Y)\neq\emptyset$. Since $\fix_{q_0}(\homeo)$ is closed in $\topspace$, we see that $\interior_Y(Y\cap\fix_{q_0}(\homeo))\neq\emptyset$. Let $V$ be an open subset of $\topspace$, such that $\emptyset\neq V\cap Y\subset Y\cap\fix_{q_0}(\homeo)$. In particular, $V\cap Y\subset \fix_{q_0}(\homeo)$, and since $Y$ is open in $\topspace$, this shows that $V\cap Y\subset\fix_{q_0}^\intsymbol(\homeo)$. Since $Y\cap(V\cap Y)\neq\emptyset$, we conclude that $Y\cap\fix_{q_0}^\intsymbol(\homeo)\neq\emptyset$, which contradicts that $Y\cap\fixqint=\emptyset$ for all $q\geq 1$ by construction.

The second, third and fourth part are immediate from the first part and \eqref{e:internal_structure_three}.
\end{proof}

\subsection*{Acknowledgements}
This work was supported by a visitor's grant of the Netherlands Organisation for Scientific Research (NWO).

\end{document}